\documentclass[final,onefignum,onetabnum]{siamart190516}
\usepackage{amsmath}
\usepackage{amssymb}

\usepackage[algo2e,vlined,ruled]{algorithm2e}

\usepackage{amsmath}

\newcommand{\nn}{\nonumber}

\renewcommand{\theequation}{\arabic{section}.\arabic{equation}}
\newcommand{\be}{\begin{equation}}
\newcommand{\ee}{\end{equation}}
\newcommand{\ba}{\begin{array}}
\newcommand{\ea}{\end{array}}
\newcommand{\bea}{\begin{eqnarray}}
\newcommand{\eea}{\end{eqnarray}}
\newcommand{\beas}{\begin{eqnarray*}}
\newcommand{\eeas}{\end{eqnarray*}}

\newtheorem{remark}{Remark}[section]

 \newcommand{\bx}{{\bf x} }

 \newcommand{\C}{\mathcal{C}}
 \newcommand{\R}{\mathcal{R}}
  \newcommand{\Z}{\mathcal{Z}}
 \newcommand{\cP}{\mathcal{P}}
 \newcommand{\M}{\mathcal{M}}
 \newcommand{\st}{\mathrm{s.t.\;} }
\newcommand{\iprod}[2]{\langle #1, #2\rangle}

\newcommand{\rn}[1]{\uppercase\expandafter{\romannumeral #1}}

\newcommand{\diag}{\mathrm{diag}}
\newcommand{\grad}{\mathrm{grad}\,}

\usepackage[normalem]{ulem}

\usepackage{bm}
\usepackage{multirow}
\usepackage{graphicx}
\usepackage{subfigure}
\usepackage{algorithmic}

\begin{document}

\title{Ground states and their characterization of spin-F
Bose-Einstein condensates \thanks{Submitted to the editors
    DATE.
\funding{The work of Z. Wen is supported in part by the NSFC grants 11421101 and 11831002, and by the National Basic Research Project under the grant 2015CB856002. The work of Y. Cai is supported in part by the NSFC grants 11771036, U1530401 and 91630204. The work of X. Wu is supported in part by the NSFC grant 91730302, and by Shanghai Science and Technology Commission Grant 17XD1400500.}}}

\author{Tonghua Tian \thanks{Yuanpei College, Peking University, CHINA (\email{fairyt@pku.edu.cn}).}
 \and 
 Yongyong Cai \thanks{Beijing Computational Science Research Center, CHINA (\email{yongyong.cai@csrc.ac.cn}).}
 \and 
 Xinming Wu \thanks{Shanghai Key Laboratory for Contemporary Applied Mathematics, School of Mathematical Sciences, Fudan University, CHINA (\email{wuxinming@fudan.edu.cn}).}
  \and Zaiwen Wen \thanks{Beijing International Center for Mathematical 
 Research, Peking University, CHINA (\email{wenzw@pku.edu.cn}).}
 }
\date{}

\headers{Ground states and their characterization of spin-F Bose-Einstein 
condensates}{T. Tian, Y. Cai, X. Wu, and Z. Wen}

\maketitle

\begin{abstract}
The computation of the ground states of spin-$F$ Bose-Einstein condensates (BECs) can 
be formulated as an energy minimization problem with two quadratic constraints. We 
discretize the energy functional and constraints using the Fourier pseudospectral schemes 
and view the discretized  problem as an optimization problem on manifold. Three different 
types of retractions to the manifold are designed. They enable us to apply various 
optimization methods on manifold to solve the problem. Specifically, an adaptive 
regularized Newton method is used together with a cascadic multigrid technique to 
accelerate the convergence. 
According to our limited knowledege, our method is the first applicable algorithm for BECs with an arbitrary integer spin, including the complicated 
spin-3 BECs.
Extensive numerical results on ground 
states of spin-1, spin-2 and spin-3 BECs with diverse interaction and optical lattice 
potential in one/two/three dimensions are reported to show the efficiency of our method 
and to demonstrate some interesting physical phenomena.
\end{abstract}

\begin{keywords} Gross-Pitaevskii theory,  spinor condensates, spin-2
 ground state, spin-3 ground state, energy minimization
\end{keywords}

\begin{AMS}35Q55, 35A01, 81Q99
\end{AMS}

\section{Introduction}
\label{sect:intro}
\setcounter{equation}{0}

Bose-Einstein condensate (BEC), first predicted by A. Einstein based on S. N. Bose's 
work, refers to the state of matter in which part of the bosons occupy the same quantum 
state at extremely low temperature. The earliest experimental observations of BEC were 
announced in 1995 \cite{Obse1,Obse2,Obse3} and have attracted numerous researchers 
into the study of condensates of dilute gases ever since 
\cite{Anderson1,Dalfovo1,Fetter1,Leggett1,Morsch1,Ozeri1,Posa1}.
While in early experiments the spin degrees of freedom are frozen due to the magnetic trapping, the experimental 
realizations of spin-1 and spin-2 condensates have been achieved later by optical confinements
\cite{Barrett,Gorl,Mies,Stam1,Sten} and revealed 
various exciting phenomena absent in single-component condensates.

Numerous theoretical studies of spinor condensates have been carried out after the 
experimental achievement \cite{Ho,Law,Ohmi,Stam2}.
At zero temperature, a spin-$F$ ($F=1,2,\ldots$) BEC is 
described by a $2F+1$ vector wave function
$\Phi=(\phi_{F},\cdots,\phi_{-F})^T\in \C^{2F+1}$
 and a generalized coupled Gross-Pitaevskii 
equation (GPE). Three important invariants of it are the mass of the wave function, 
the magnetization and the energy per particle. A fundamental problem in BEC is to find the 
condensate stationary states, which is obtained by minimizing the Gross-Pitaevskii 
(GP) energy functional subject to the conservation of total mass and magnetization.

Different numerical methods have been proposed in the literature to compute the ground 
state of a spin-1 BEC \cite{spin1-1, spin1-3, spin1-4, spin1-5, Bao1}. Among them, a very 
popular method is the imaginary time method combined with a proper discretization scheme 
to evolve the resulted gradient flow equation under the normalization of the wave function 
\cite{it-1, Bao-Du-2004, Bao1, spin1-1}. To apply the normalized gradient flow method to compute the 
ground state of a spin-$F$ BEC, $2F+1$ projection constants have to be determined in the 
normalization step, while only two normalization conditions (i.e., the two constraints) are 
given. In the literature, this method is applied to compute the ground state of a spin-1 BEC 
through the introduction of a random variable \cite{spin1-4, spin1-5} or a third normalization 
condition \cite{Bao1}. Recently, a projection gradient method \cite{spin1-1, spin2-1} has 
been proposed to compute ground states of spin-1 and spin-2 BEC, where a continuous 
normalized gradient flow (CNGF) was discretized by the Crank-Nicolson finite difference 
(CNFD) method with a proper and very special way to deal with the nonlinear terms. 
This scheme is proved to be mass- and magnetization-conservative and 
energy-diminishing in the discretized level. However, a fully nonlinear coupled system has to be solved at each time step.

Most of the existing numerical methods for computing the ground states of spinor BEC 
evolve from the gradient flow method, and thus converge at most linearly and/or require 
to solve a large scale linear system per iteration, which leads to quite expensive 
computational cost. Most of them are specially designed for spin-1 or spin-2 
BEC, but the spin-3 cases are rarely discussed. 
Meanwhile, over the last decade, some advanced optimization methods 
have been developed for solving minimization problems on matrix manifolds, such as 
the Riemannian Newton methods and trust-region methods \cite{Absil, AbsilRTR} with 
superlinear or quadratic convergence rate. The aim of this paper is to explore a new way
to compute the ground states of spinor BEC, and propose 
an efficient regularized Newton method for the general spin-$F$ cases. We first 
discretize the energy functional and the constraints with the Fourier pseudospectral 
schemes and thus approximate the original infinite dimensional problem by a finite 
dimensional minimization problem, of which the feasible region can be proven to be 
a Riemannian manifold. We give the formulas of Riemannian gradient and 
Hessian on this manifold, and then aim to apply an adaptive regularized Newton 
method to solve the Riemannian optimization problem. To improve the efficiency and 
stability, we adopt the cascadic multigrid technique and use a Riemannian gradient 
method with Barzilai-Borwein step size to compute initial points on each mesh.
Three different retractions on the manifold are proposed for the implementation of 
Riemannian optimization algorithms. The first one is the classical projective retraction, 
and the second one comes from the normalized gradient flow \cite{Bao1}. The computation 
of them relies on finding a unique zero of a single-variable function, which can be done 
quite efficiently and accurately. The third retraction is proposed as an approximation of 
the first one, with a very brief closed-form formula.
Extensive numerical experiments demonstrate that our approach can quickly compute 
an accurate approximation of the ground state, and is more stable than the classical 
Riemannian trust-region method. 
The algorithm remains effective even for the complicated 
spin-3 BEC in 3D with an optical lattice potential, for which there exists no 
applicable algorithm before.

The rest of this paper is organized as follows. Specific problem statements of spin-1, spin-2 and spin-3 BEC are given in \cref{sect:prob}. Discretizations of the energy functional and the constraints via the Fourier pseudospectral 
schemes are introduced in \cref{sect:discre}. In \cref{sect:mani}, we give some 
preliminaries on Riemannian optimization, and 
investigate the manifold structure of the feasible region. In \cref{sect:arnt}, 
we present a modified version of the adaptive regularized Newton method for 
solving the discretized optimization problem. The three retractions are 
described in \cref{sect:retr}, and detailed numerical results are reported in 
\cref{sect:result} to illustrate the efficiency and 
accuracy of our algorithm. Finally, some conclusions are given in 
\cref{sect:conclusion}.

\section{Problem Statement}
\label{sect:prob}

The specific formulation of the minimization problem for computing the ground states of  
spin-1, spin-2 and spin-3 BEC is stated as follows:

{\bf Spin-1}. For a spin-1 BEC, the GP energy functional for the spin-1 wave function %$\Phi=(\phi_1,\phi_0,\phi_{-1})^T\in\C^3$
is given by
\begin{align}
E(\Phi(\cdot))&=\int_{\R^d
}\bigg\{\sum_{l=-1}^{1} \left(\frac{1}{2}|\nabla
\phi_l|^2+(V(\bx)-pl+ql^2)|\phi_l|^2\right)
+\frac{\beta_0}{2}|\Phi|^4+\frac{\beta_1}{2}|\mathbf{F}|^2\bigg\}\; d\bx,\label{eq:energy:spin1}
\end{align}
where $\bx=x$ in 1D, $\bx=(x,y)^T$ in 2D and $\bx=(x,y,z)^T$ in 3D, $V(\bx)$ is the external confining potential, $p$ and $q$ are the linear and quadratic Zeeman energy shifts, respectively.
$\beta_0$ is the density dependent interaction strength between the particles and $\beta_1$ is the spin dependent interaction strength,
$\mathbf{F}:=\mathbf{F}(\Phi)=(F_x,F_y,F_z)^T\in\R^3$ is the spin vector given by
\be\label{eq:F}
F_x=\Phi^*\mathrm{f}_x\Phi,\quad F_y=\Phi^*\mathrm{f}_y\Phi,\quad F_z=\Phi^*\mathrm{f}_z\Phi,
\ee
where $\Phi^*=\overline{\Phi}^T$ is the conjugate transpose and $\mathrm{f}_\alpha$ ($\alpha=x,y,z$) are the 3-by-3 spin-1 matrices
\be\label{eq:spin1m}
\mathrm{f}_x=\frac{1}{\sqrt{2}}\begin{pmatrix}0&1&0\\
1&0&1\\
0&1&0\end{pmatrix},\quad \mathrm{f}_y=\frac{i}{\sqrt{2}}\begin{pmatrix}0&-1&0\\
1&0&-1\\
0&1&0\end{pmatrix},\quad\mathrm{f}_z=\begin{pmatrix}1&0&0\\
0&0&0\\
0&0&-1\end{pmatrix},
\ee
and $i=\sqrt{-1}$ is the imaginary unit.
In detail, the components of spin vector $\mathbf{F}$ can be written explicitly as:  $F_z=|\phi_1|^2-|\phi_{-1}|^2$,
\be\label{eq:spin1v}
\begin{split}
&F_x=\frac{1}{\sqrt{2}}\left[\overline{\phi}_1\phi_0+\overline{\phi}_0
(\phi_1+\phi_{-1})+\overline{\phi}_{-1}\phi_0\right],\quad F_y=\frac{i}{\sqrt{2}}\left[-\overline{\phi}_1\phi_0+\overline{\phi}_0
(\phi_1-\phi_{-1})+\overline{\phi}_{-1}
\phi_0\right].
\end{split}
\ee
%The energy \revise{functional \eqref{eq:energy:spin1}} is usually subject to the following two constraints, i.e. the {\sl mass} (or {\sl normalization}) as
%\be\label{eq:norm:spin1}
%N(\Phi(\cdot)):=\|\Phi(\cdot)\|^2=\int_{\Bbb R^d} \sum_{l=-1}^1 |\phi_l(\bx)|^2\,d\bx=1,
%\ee
% and the {\sl magnetization} (with $M\in[-1,1]$) as
%\be\label{eq:mag:spin1}
%M(\Phi(\cdot)):=\int_{\Bbb R^d}\sum_{l=-1}^1 l |\phi_l(\bx)|^2\,d\bx
%=M.
%\ee
%The ground state $\Phi_g(\bx)$ of the spin-1 BEC described by the energy \eqref{eq:energy:spin1}
%is obtained from the minimization of
%the energy functional subject to the conservation of total mass
%and magnetization:
%\begin{quote}
%  Find $\left(\Phi_g \in S_M\right)$ such that
%\end{quote}
%  \begin{equation}\label{eq:minimize:spin1}
%    E_g := E\left(\Phi_g\right) = \min_{\Phi \in S_M}
%    E\left(\Phi\right),
%  \end{equation}
%\noindent where the nonconvex set $S_M$ is defined as
%\be \label{eq:cons:sec4}
%S_M=\left\{\Phi=(\phi_1,\phi_0,\phi_{-1})^T\ |\ \|\Phi\|=1, \
%\int_{{\Bbb R}^d} \left[|\phi_1(\bx)|^2
%-|\phi_{-1}(\bx)|^2\right]\,d\bx=M, \ E(\Phi)<\infty\right\}. \ee

{\bf Spin-2}. For a spin-2 BEC, %the system is characterized by a vector wave function 
%$\Phi=(\phi_{2},\phi_{1},\phi_0,\phi_{-1},\phi_{-2})^T\in \C^5$,
%$\Phi=(\phi_{2},\cdots,\phi_{-2})^T\in \C^5$,
the
GP energy is given by
\begin{align}
E(\Phi(\cdot))&=\int_{\R^d
}\bigg\{\sum_{l=-2}^{2} \left(\frac{1}{2}|\nabla
\phi_l|^2+(V(\bx)-pl+ql^2)|\phi_l|^2\right)
+\frac{\beta_0}{2}|\Phi|^4+\frac{\beta_1}{2}|\mathbf{F}|^2
+\frac{\beta_2}{2}|A_{00}|^2\bigg\}\; d\bx,\label{eq:energy:spin2}
\end{align}
where $\beta_2$ is the spin-singlet interaction strength and all the other parameters $p,q,\beta_0,\beta_1$ are the same as those in the spin-1 case,
$\mathbf{F}:=\mathbf{F}(\Phi)=(F_x,F_y,F_z)^T\in\R^3$ is the spin vector defined by \cref{eq:F},
with  $\mathrm{f}_\alpha$ ($\alpha=x,y,z$) given by the 5-by-5 spin-2 matrices
\be
\mathrm{f}_x=\begin{pmatrix}0&1&0&0&0\\
1&0&\sqrt{\frac{3}{2}}&0&0\\
0&\sqrt{\frac{3}{2}}&0&\sqrt{\frac{3}{2}}&0\\
0&0&\sqrt{\frac{3}{2}}&0&1\\
0&0&0&1&0\end{pmatrix},\quad \mathrm{f}_y=i\begin{pmatrix}0&-1&0&0&0\\
1&0&-\sqrt{\frac{3}{2}}&0&0\\
0&\sqrt{\frac{3}{2}}&0&-\sqrt{\frac{3}{2}}&0\\
0&0&\sqrt{\frac{3}{2}}&0&-1\\
0&0&0&1&0\end{pmatrix}
\ee
and
\be
\mathrm{f}_z=\text{diag}(2,1,0,-1,-2).
\ee
Therefore, the spin vector $\mathbf{F}$ can be written explicitly
%$\mathbf{F}=\mathbf{F}(\Phi)=(F_x,F_y,F_z)^T$ with $F_\alpha=\Phi^*\mathrm{f}_\alpha\Phi$ ($\alpha=x,y,z$) given as
\be\label{eq:F:spin2}\begin{split}
&F_x=\overline{\phi}_2\phi_1+\overline{\phi}_1\phi_2+\overline{\phi}_{-1}\phi_{-2}+\overline{\phi}_{-2}\phi_{-1}+\frac{\sqrt{6}}{2}(\overline{\phi}_1\phi_0+
\overline{\phi}_0\phi_1+\overline{\phi}_0\phi_{-1}+\overline{\phi}_{-1}\phi_0),\\
&F_y=i\left[\overline{\phi}_1\phi_2-\overline{\phi}_2\phi_1+\overline{\phi}_{-2}\phi_{-1}-\overline{\phi}_{-1}\phi_{-2}+\frac{\sqrt{6}}{2}(
\overline{\phi}_0\phi_1-\overline{\phi}_1\phi_0+\overline{\phi}_{-1}\phi_0-\overline{\phi}_0\phi_{-1})\right],\\
&F_z=2|\phi_2|^2+|\phi_1|^2-|\phi_{-1}|^2-2|\phi_{-2}|^2.
\end{split}
\ee
%with $\mathbf{F}\cdot\mathbf{f}=F_x\mathrm{f}_x+F_y\mathrm{f}_y+F_z\mathrm{f}_z$. 
Define the matrix
%The matrix $\mathbf{A}$ is
\be
\mathbf{A}=\frac{1}{\sqrt{5}}\begin{pmatrix}0&0&0&0&1\\
0&0&0&-1&0\\
0&0&1&0&0\\
0&-1&0&0&0\\
1&0&0&0&0\end{pmatrix},
\ee
then $A_{00}:=A_{00}(\Phi)=\Phi^T \mathbf{A}\Phi$ %with
can be expressed as
\be\label{eq:A:spin2}
A_{00}=\frac{1}{\sqrt{5}}(2\phi_2\phi_{-2}-2\phi_{1}\phi_{-1}+\phi_0^2).
\ee
%The energy \revise{functional} \eqref{eq:energy:spin2} is usually subject to the following two constraints, i.e. the {\sl mass} (or {\sl normalization}) as
%\be\label{eq:norm:spin2}
%N(\Phi(\cdot)):=\|\Phi(\cdot)\|^2=\int_{\Bbb R^d}\sum_{l=-2}^2|\phi_l(\bx)|^2\,d\bx=1,
%\ee
% and the {\sl magnetization} (with $M\in[-2,2]$) as
%\be\label{eq:mag:spin2}
%M(\Phi(\cdot)):=\int_{\Bbb R^d}\sum_{l=-2}^2l|\phi_l(\bx)|^2\,d\bx
%=M.
%\ee
%The ground state $\Phi_g(\bx)$ of the spin-2 BEC described by the energy \eqref{eq:energy:spin2}
%is obtained from the minimization of
%the energy functional subject to the conservation of total mass
%and magnetization:
%\begin{quote}
%  Find $\left(\Phi_g \in S_M\right)$ such that
%\end{quote}
%  \begin{equation}\label{eq:minimize:spin2}
%    E_g := E\left(\Phi_g\right) = \min_{\Phi \in S_M}
%    E\left(\Phi\right),
%  \end{equation}
%\noindent where the nonconvex set $S_M$ is defined as
%\be \label{eq:cons:spin2}
%S_M=\left\{\Phi\in\C^5\ |\ \|\Phi\|=1, \
%\int_{{\Bbb R}^d} \sum_{l=-2}^2l | \phi_l(\bx) |^2=M, \ E(\Phi)<\infty\right\}. \ee

{\bf Spin-3}.
For a spin-3 BEC, %the system is described by the vector wave function $\Phi=(\phi_3,\ldots,\,\phi_{-3})^T\in \C^7$, and the
the GP energy is given by
\begin{align}
E(\Phi(\cdot))&=\int_{\R^d
}\bigg\{\sum_{l=-3}^{3} \left(\frac{1}{2}|\nabla
\phi_l|^2+(V(\bx)-pl+ql^2)|\phi_l|^2\right)
+\frac{\beta_0}{2}|\Phi|^4+\frac{\beta_1}{2}|\mathbf{F}|^2
+\frac{\beta_2}{2}|A_{00}|^2\nonumber\\
&\qquad\qquad+\frac{\beta_3}{2}\sum_{l=-2}^2|A_{2l}|^2\bigg\}\; d\bx,\label{eq:energy:spin3}
\end{align}
where $\beta_3$ is the spin-quintet interaction strength, and all the other parameters $p,q,\beta_0,\beta_1,\beta_2$ are the same as those in the spin-1,2 cases.
$\mathbf{F}:=\mathbf{F}(\Phi)=(F_x,F_y,F_z)^T\in\R^3$ is the spin vector defined by \eqref{eq:F},
%\be
%F_x=\Phi^*\mathrm{f}_x\Phi,\quad F_y=\Phi^*\mathrm{f}_y\Phi,\quad F_z=\Phi^*\mathrm{f}_z\Phi,
%\ee
%where  $\mathrm{f}_\alpha$ ($\alpha=x,y,z$) are the 7-by-7 spin-3 matrices
with  $\mathrm{f}_\alpha$ ($\alpha=x,y,z$) given by the 7-by-7 spin-3 matrices
\be
\mathrm{f}_x=\begin{pmatrix}0&\sqrt{3/2}&0&0&0&0&0\\
\sqrt{3/2}&0&\sqrt{5/2}&0&0&0&0\\
0&\sqrt{5/2}&0&\sqrt{3}&0&0&0\\
0&0&\sqrt{3}&0&\sqrt{3}&0&0\\
0&0&0&\sqrt{3}&0&\sqrt{5/2}&0\\
0&0&0&0&\sqrt{5/2}&0&\sqrt{3/2}\\
0&0&0&0&0&\sqrt{3/2}&0\end{pmatrix},
\ee
\be
\mathrm{f}_y=i\begin{pmatrix}0&-\sqrt{3/2}&0&0&0&0&0\\
\sqrt{3/2}&0&-\sqrt{5/2}&0&0&0&0\\
0&\sqrt{5/2}&0&-\sqrt{3}&0&0&0\\
0&0&\sqrt{3}&0&-\sqrt{3}&0&0\\
0&0&0&\sqrt{3}&0&-\sqrt{5/2}&0\\
0&0&0&0&\sqrt{5/2}&0&-\sqrt{3/2}\\
0&0&0&0&0&\sqrt{3/2}&0\end{pmatrix}
\ee
and
\be
\mathrm{f}_z=\text{diag}(3,2,1,0,-1,-2,-3).
\ee
The spin vector $\mathbf{F}$ can  be written explicitly
\begin{equation*}\begin{split}
&F_+=F_x+iF_y=\sqrt{6}\overline{\psi}_3\psi_2+\sqrt{10}\overline{\psi}_2\psi_1+2\sqrt{3}\overline{\psi}_1\psi_0+2\sqrt{3}\overline{\psi}_0\psi_{-1}+
\sqrt{10}\overline{\psi}_{-1}\psi_{-2}+\sqrt{6}\overline{\psi}_{-2}\psi_{-3},\\
&F_z=3|\psi_3|^2+2|\psi_2|^2+|\psi_1|^2-|\psi_{-1}|^2-2|\psi_{-2}|^2-3|\psi_{-3}|^2.
\end{split}
\end{equation*}
%with $\mathbf{F}\cdot\mathbf{f}=F_x\mathrm{f}_x+F_y\mathrm{f}_y+F_z\mathrm{f}_z$. 
%The matrices $\mathbf{A}$ and $\mathbf{A}_0$  are defined as
Define the matrices
\begin{equation*}
\mathbf{A}=\frac{1}{\sqrt{7}}\begin{pmatrix}
0&0&0&0&0&0&1\\
0&0&0&0&0&-1&0\\
0&0&0&0&1&0&0\\
0&0&0&-1&0&0&0\\
0&0&1&0&0&0&0\\
0&-1&0&0&0&0&0\\
1&0&0&0&0&0&0\end{pmatrix},\ \ 
\mathbf{A}_0=\frac{1}{\sqrt{7}}\begin{pmatrix}
0&0&0&0&0&0&\frac{5}{2\sqrt{3}}\\
0&0&0&0&0&0&0\\
0&0&0&0&\frac{-\sqrt{3}}{2}&0&0\\
0&0&0&\sqrt{\frac{2}{3}}&0&0&0\\
0&0&\frac{-\sqrt{3}}{2}&0&0&0&0\\
0&0&0&0&0&0&0\\
\frac{5}{2\sqrt{3}}&0&0&0&0&0&0\end{pmatrix}
\end{equation*}
and $\mathbf{A}_{l}=(a_{l,jk})_{7\times7}$ ($l=\pm1,\pm2$), where $a_{l,jk}$ is zero except for those $j+k=8-l$. 
For the simplicity of notations, we denote $\vec{a}_l=(a_{l,1(7-l)},
a_{l,2(6-l)},\ldots,a_{l,(7-l)1})^T\in\R^{7-l}$ for $l=1,2$ and $\vec{a}_l=(a_{l,(1-l)7},
a_{l,(2-l)6}, \ldots,a_{l,7(1-l)})^T\in\R^{7+l}$ for $l=-1,-2$ with
\begin{align*}
&\vec{a}_{\pm1}=\frac{1}{\sqrt{7}}\left(\frac{5}{2\sqrt{3}},-\frac{\sqrt{5}}{2},\frac{1}{\sqrt{6}},\frac{1}{\sqrt{6}},-\frac{\sqrt{5}}{2},\frac{5}{2\sqrt{3}}\right)^T,\quad
\vec{a}_{\pm2}=\frac{1}{\sqrt{7}}\left(\sqrt{\frac{5}{6}},-\sqrt{\frac{5}{3}},\sqrt{2},-\sqrt{\frac{5}{3}},\sqrt{\frac{5}{6}}\right)^T.
\end{align*}
Then $A_{00}:=A_{00}(\Psi)=\Psi^T \mathbf{A}\Psi$ and $A_{2l}:=A_{2l}(\Psi)=\Psi^T\mathbf{A}_l\Psi$ can be expressed as
\begin{align}\label{eq:A:spin3}
A_{00}=&\frac{1}{\sqrt{7}}(2\psi_3\psi_{-3}-2\psi_{2}\psi_{-2}+2\psi_1\psi_{-1}-\psi_0^2),\\ A_{20}=&\frac{1}{\sqrt{21}}(5\psi_3\psi_{-3}-3\psi_1\psi_{-1}+\sqrt{2}\psi_0^2),\\
A_{2,\pm1}=&\frac{1}{\sqrt{21}}(5\psi_{\pm3}\psi_{\mp2}-\sqrt{15}\psi_{\pm2}\psi_{\mp1}+\sqrt{2}\psi_{\pm1}\psi_0),\\A_{2,\pm2}=&\frac{1}{\sqrt{21}}(\sqrt{10}\psi_{\pm3}\psi_{\mp1}-\sqrt{20}\psi_{\pm2}\psi_{0}+\sqrt{2}\psi_{\pm1}^2).
\end{align}

For computing the ground state of a spin-$F$ BEC, the energy functional $E(\Phi(\cdot))$ is usually subject to the following two constraints, i.e. the {\sl mass} (or {\sl normalization}) as
\be\label{eq:norm:spin3}
N(\Phi(\cdot)):=\|\Phi(\cdot)\|^2=\int_{\Bbb R^d}\sum_{l=-F}^F|\phi_l(\bx)|^2\,d\bx=1,
\ee
 and the {\sl magnetization} (with $M\in[-F,F]$) as
\be\label{eq:mag:spin3}
M(\Phi(\cdot)):=\int_{\Bbb R^d}\sum_{l=-F}^Fl|\phi_l(\bx)|^2\,d\bx
=M.
\ee
The ground state $\Phi_g(\bx)$ is obtained from the minimization of
the energy functional subject to the conservation of total mass
and magnetization:
\begin{quote}
  Find $\left(\Phi_g \in S_M\right)$ such that
\end{quote}
  \begin{equation}\label{eq:minimize:spin3}
    E_g := E\left(\Phi_g\right) = \min_{\Phi \in S_M}
    E\left(\Phi\right),
  \end{equation}
\noindent where the nonconvex set $S_M$ is defined as
\be \label{eq:cons:spin3}
S_M=\left\{\Phi=(\phi_F,\ldots,\phi_{-F})^T\in\C^{2F+1}\ |\ \|\Phi\|=1, \
\int_{{\Bbb R}^d} \sum_{l=-F}^Fl | \phi_l(\bx) |^2=M, \ E(\Phi)<\infty\right\}. \ee

For $M=\pm F$ in the spin-$F$ BEC, the constraints ensure only one component $\phi_{\pm F}$ is nonzero, and \eqref{eq:minimize:spin3} reduces to the single component BEC ground state problems which have been considered.   Therefore, we will assume $|M|<F$ for the spin-$F$ BEC ground states in the rest part of the paper.

\subsection{Notoations}

Given $X\in \C^{m\times n}$,   $\bar X, X^T, X^*$ and  $\Re(X)$ denote the complex conjugate, the transpose, the complex conjugate transpose, and the real part of $X$, respectively. 
%The Euclidean inner product between two matrices $X,\, Y\in \C^{m\times n}$ is defined as $\iprod{X}{Y}:=\sum_{jk}\bar X_{jk}Y_{jk}=\mathrm{tr}(X^*Y)$. The trace of $X$, i.e., the sum of the diagonal elements of $X\in\C^{n\times n}$, is denoted by $\mathrm{tr}(X)$. For a given vector $d\in\C^n$, the operator $\diag(d)$ returns a square matrix in $\C^{n\times n}$ with the elements of $d$ on the main diagonal, while $\diag(X)$ gives a column vector in $\C^n$ consisting of the main diagonal of $X$. 
The trace of $X$, i.e., the sum of the diagonal elements of $X\in\C^{n\times n}$, is denoted by $\mathrm{tr}(X)$. For a given vector $d\in\C^n$, the operator $\diag(d)$ returns a square matrix in $\C^{n\times n}$ with the elements of $d$ on the main diagonal, while $\diag(X)$ gives a column vector in $\C^n$ consisting of the main diagonal of $X$. The Euclidean inner product between two matrices $X,\, Y\in \C^{m\times n}$ is defined as $\iprod{X}{Y}:=\sum_{jk} X_{jk}\bar{Y}_{jk}=\mathrm{tr}(Y^*X)$.

\section{Discretization Schemes}
\label{sect:discre}

%In this section, we introduce discretization of the energy functional \eqref{eq:energy:spin2} and constraints \eqref{eq:norm:spin2}- \eqref{eq:mag:spin2} in the constrained minimization problem \eqref{eq:minimize:spin2} for the Spin-2 case. It is similar and much easier to deal with the Spin-1 case.
In this section, we introduce discretization of the energy functional \eqref{eq:energy:spin3} and constraints \eqref{eq:norm:spin3}- \eqref{eq:mag:spin3} in the constrained minimization problem \eqref{eq:minimize:spin3} for the spin-3 case. It is similar and much easier to deal with the spin-1 and spin-2 cases.
Due to the external trapping potential, the ground state of \eqref{eq:minimize:spin3} decays exponentially as $|\mathbf{x}|\to \infty$.  Thus we can truncate the energy functional and constraints from the whole space $\R^d$ to a bounded computational domain $U$ which is chosen large enough such that the truncation error is negligilbe with periodic boundary condition. Then we approximate spatial derivatives via the Fourier pseudospectral (FP) method and the integrals via the composite trapezoidal quadrature. For simplicity of notation, we only present the FP discretization in 1D. Extensions to 2D and 3D are straightforward for tensor grids and the details are omitted here for brevity.

For $d=1$, we take a bounded interval $U=(a,b)$. Let $h=(b-a)/n$ be the spatial mesh size with $n$ an even positive integer and denote $x_j=a+jh$ for $j=0,1,\cdots, n$.
Let $\phi_{jl}$ be the numerical approximation of $\phi_l(x_j)$  for $j=0,1,\cdots,n$ and $l=3,\cdots,-3$ 
satisfying $\phi_{0l}=\phi_{nl}$ and denote 
$X=( \sqrt{h}\phi_{jl} )\in\C^{n\times 7}$ ($j=0,1,\cdots, n-1$, $l=3,\cdots,-3$).

\begin{eqnarray}
E(\Phi) 
&=&\sum_{j=0}^{n-1}\int_{x_j}^{x_{j+1}} \bigg\{\sum_{l=-3}^{3} \left(-\frac{1}{2} \bar\phi_l \;\partial_{xx}\phi_l +(V(x)-pl+ql^2)|\phi_l|^2\right) \nn \\
&& \qquad\qquad\qquad +\frac{\beta_0}{2}|\Phi|^4+\frac{\beta_1}{2}|\mathbf{F}|^2 +\frac{\beta_2}{2}|A_{00}|^2 
+ \frac{\beta_3}{2}\sum_{l=-2}^{2} |A_{2l}|^2\bigg\}\; dx  \\
&\approx& h \sum_{j=0}^{n-1}  \bigg\{\sum_{l=-3}^{3} \left(-\frac12\bar\phi_{l}(x_j) \;\partial_{xx}^f\phi_l\Big|_{j} +(V(x_j)-pl+ql^2)|\phi_{l}(x_j)|^2\right)\nn \\
&& \qquad\qquad\qquad +\frac{\beta_0}{2}|\Phi(x_j)|^4+\frac{\beta_1}{2}|\mathbf{F}(x_j)|^2+\frac{\beta_2}{2}|A_{00}(x_j)|^2 
+ \frac{\beta_3}{2}\sum_{l=-2}^{2} |A_{2l}(x_j)|^2\bigg\}\label{eq:discrete:energy0}
\end{eqnarray}
where the Fourier pseudospectral differential operator is given as
\begin{equation}\label{eq:discrete:laplace}
\partial_{xx}^f\phi\Big|_{j} = -\frac1n \sum_{p=-n/2}^{n/2-1}\lambda_p^2\tilde\phi_{pl} e^{i\frac{2\pi jp}{n}},
\end{equation}
with
\begin{equation}\label{eq:discrete:fourier}
\tilde\phi_{pl} = \sum_{j=0}^{n-1}\phi_{l}(x_j) e^{-i\frac{2\pi jp}{n}},\quad \lambda_p=\frac{2\pi p}{b-a},\quad p=-\frac n2,\cdots,\frac n2-1.
\end{equation}
Introduce $\mathrm{V}=\diag(V(x_0), \cdots, V(x_{n-1}))$,  $B=\diag(b)$ with $b=(b_3,b_2,\ldots,b_{-3})^T$ ($b_l=-pl + ql^2$, $l=3,\cdots,-3$), $\Lambda=\diag(\lambda_{-\frac n2}^2, \cdots, \lambda_{ \frac n2-1}^2)$,
 and $C=(c_{jp})\in\C^{n\times n}$ with entries $c_{jp}=e^{-i\frac{2\pi jp}{n}}$ for $j=0,\cdots,n-1$ and $p=-\frac n2,\cdots,\frac n2 -1$.
Plugging \eqref{eq:discrete:laplace} and \eqref{eq:discrete:fourier} into \eqref{eq:discrete:energy0}, and replacing $\phi_l(x_j)$ by $\phi_{jl}$, we get 
the finite dimensional approximation to the energy functional defined as
%\begin{eqnarray}
%E_h(X) &=& h\Big\{\, \frac{1}{2}\mathrm{tr}(X^* L X) + \mathrm{tr}(X^* V X) + \mathrm{tr}(X B X^*) \nn\\
%&&  + \frac{\beta_0}{2}\rho^T \rho + \frac{\beta_1}{2}\sum_{\alpha=x,y,z}F_\alpha^T F_\alpha + \frac{\beta_2}{2}A_{00}^* A_{00}\,\Big\},\label{eq:discrete:energy}
%\end{eqnarray}
\begin{eqnarray}
E_h(X) &=& \frac{1}{2}\mathrm{tr}(X^* L X) + \mathrm{tr}(X^* \mathrm{V} X) + \mathrm{tr}(X B X^*) \nn\\
&&  + \frac{\beta_0}{2h}\rho^T \rho + \frac{\beta_1}{2h}\sum_{\alpha=x,y,z} 
\mathrm{F}_\alpha^T \mathrm{F}_\alpha + \frac{\beta_2}{2h}\mathrm{A}_{00}^* 
\mathrm{A}_{00} + \frac{\beta_3}{2h}\sum_{l=-2}^{2} \mathrm{A}_{2l}^* \mathrm{A}_{2l},\label{eq:discrete:energy}
\end{eqnarray}
where $L=C^*\Lambda C$ is the matrix representation of the discrete negative Laplace operator and
%\begin{equation*}
%\rho(X)=\mathrm{diag}(X X^*), \quad F_\alpha = \mathrm{diag}(X f_\alpha^T X^*),\ \alpha=x,y,z,\quad A_{00}=\mathrm{diag}(X A X^T),
%\end{equation*}
\begin{eqnarray}
&&\rho=\mathrm{diag}(X X^*), \quad\ \mathrm{F}_\alpha = \mathrm{diag}(X f_\alpha^T X^*),\ \alpha=x,y,z,\\
&&\mathrm{A}_{00}=\mathrm{diag}(X \mathbf{A} X^T),\quad \mathrm{A}_{2l} = \diag(X \mathbf{A}_l X^T),\ l=-2,\cdots,2
\end{eqnarray}
are column vectors. In fact, the first term in \eqref{eq:discrete:energy} can be computed efficiently at cost $O(n\ln n)$ through the discretized Fourier 
transform (DFT).

Similarly, let $D = \diag(3,\cdots,-3)$, the constraints \eqref{eq:norm:spin3}- \eqref{eq:mag:spin3} can be truncated and discretized as 
\begin{eqnarray}
&& N(\Phi(\cdot))  %\approx \sum_{j=0}^{n-1}\int_{x_j}^{x_{j+1}}\sum_{l=-3}^{3} |\phi_l(x)|^2 dx
 \approx h \sum_{j=0}^{n-1}\sum_{l=-3}^{3} |\phi_{jl}|^2  = \mathrm{tr}(X^* X) = 1,\\
&& M(\Phi(\cdot)) % \approx \sum_{j=0}^{n-1}\int_{x_j}^{x_{j+1}}\sum_{l=-3}^{3} l |\phi_l(x)|^2 dx
 \approx h \sum_{j=0}^{n-1}\sum_{l=-3}^{3} l |\phi_{jl}|^2 = \mathrm{tr}(X^* X D) = M,
\end{eqnarray}
which immediately implies that the set $S_M$ can be discretized as
\begin{equation}
S_h=\{ X\in \C^{n\times 7}\ | \ \mathrm{tr}(X^* X) = 1,\ \mathrm{tr}(X^* X D) = M, \ E_h(X) < \infty \}.
\end{equation}
Hence, the original problem \eqref{eq:minimize:spin3} with $d=1$ can be approximated by the discretized minimization problem via the FP
discretization:
\begin{equation}\label{eq:minimize:discrete}
%X_g = \arg\min_{X\in S_h} E_h(X).
E_{h,g} := E_h(X_g) = \min_{X\in S_h} E_h(X).
\end{equation}

To solve the discrete minimization problem \eqref{eq:minimize:discrete}, it is often necessary to compute the gradient and  Hessian matrix of the discrete energy $E_h(X)$.
The second-order Taylor expansion of $E_h(X)$ can be expressed as
\begin{eqnarray}
E_h(X+\Delta X) = E_h(X) + \Re\iprod{\nabla E_h}{\Delta X} + \frac12\Re\iprod{(\nabla^2 E_h) \Delta X}{\Delta X} + \mathrm{h.o.t.},
\end{eqnarray}
where h.o.t. is short for  the higher-order terms. By  a simple calculation, we can get the gradient 
\begin{eqnarray}
\nabla E_h(X) &=&  LX + 2\mathrm{V}X + 2XB \nn \\ 
&&  + \frac{2\beta_0}{h}\diag(\rho)X + \frac{2\beta_1}{h}\sum_{\alpha=x,y,z}\diag(\mathrm{F}_\alpha)Xf_\alpha^T  \nn \\
&&  + \frac{2\beta_2}{h}\diag(\mathrm{A}_{00})\bar X\mathbf{A}
+ \frac{2\beta_3}{h}\sum_{l=-2}^{2}\diag(\mathrm{A}_{2l})\bar X \mathbf{A}_l,
\end{eqnarray}
and the Hessian-vector product 
\begin{eqnarray}
\nabla^2E_h(X) [Z] &=&  LZ + 2\mathrm{V}Z + 2ZB  \nn \\
&& + \frac{2\beta_0}{h}\left( \, \diag(\rho)Z + 2\diag(\Re(ZX^*))X \, \right) \nn \\
&& + \frac{2\beta_1}{h}\sum_{\alpha=x,y,z}\left(\diag(\mathrm{F}_\alpha)Zf_\alpha^T + 2\diag(\Re(Zf_\alpha^T X^* ))Xf_\alpha^T\right) \nn \\
&& + \frac{2\beta_2}{h}\left(\, \diag(\mathrm{A}_{00})\bar Z\mathbf{A} + 2\diag(Z
\mathbf{A}X^T)\bar X \mathbf{A} \,  \right)\nn \\
&& + \frac{2\beta_3}{h}\sum_{l=-2}^{2}\left(\, \diag(\mathrm{A}_{2l})\bar Z \mathbf{A}_l + 2\diag(Z \mathbf{A}_l X^T)\bar X \mathbf{A}_l \,  \right).
\end{eqnarray} 

\section{Manifold Structure}
\label{sect:mani}

In the ground state of a spin-$F$ BEC, we have $M \leftrightarrow -M \Longleftrightarrow 
\phi_l \leftrightarrow \phi_{-l}. $ Thus we only discuss the cases where $M\geq 0$. 
Express $X$ as $X = X_r + i X_i$, where $X_r, X_i\in \R^{n\times (2 F+1)}$.
%\remove{and $i$ denotes the imaginary unit.}
Let $(X_r; X_i) = (u_F, u_{F-1}, \cdots, u_0, \cdots, u_{-F+1}, u_{-F})\in\R^{2n\times(2F+1)}$ and $u\in\R^N$ be the reconstructed column vector of this matrix, where $N = 2n (2F+1)$.
%$u = (u_F; u_{F-1}; \cdots; u_0; \cdots; u_{-F+1}; u_{-F})$ 
Introduce 
\be
\Gamma=\diag (F I_{2n}, ..., I_{2n}, 0, -I_{2n}, ..., -F I_{2n})\in \R^{N \times N},\ee 
where $I_{2n}$ denotes the identity matrix of size $2n$.
Then the constraints can be discretized as 
\be \label{dcon}\M =\left\{ u\in\R^N  \mid u^T u =1, u^T \Gamma u = M
\right\},\ee
in which $0 \leq M < F.$
Define $\widetilde{E}(u) := E_h (X)$, our model problem can be formulated as
\be \label{dprob} \min \quad \widetilde{E}(u), \quad \st  u \in \M. \ee

This is a nonconvex optimization problem with constraints.
Observe that $\M$ is a level set of the function
 \be \textbf{G}(u) = \frac{1}{2}(u^T u - 1, u^T \Gamma u - M)^T.\ee
When $M \notin \Z$, $\nabla \mathbf{G}(u) = (u, \Gamma u)^T$ has full rank at every point
$u \in \M$, thus according to Proposition 3.3.3  in \cite{Absil}, $\M$ is a closed embedded 
submanifold of $\R^N$ of dimension $N-2$.
In the following discussion, we %	always assume $0 < M < F, M \notin \Z$, 
will let  $u$ be an arbitrary point on the manifold $\M$ and  assume $f$ is a
smooth real-valued function in a neighborhood of $u$.

Given a curve $\gamma(t):\R \rightarrow \M\subset\R^N$ through $u$ at $t=0$, the associated tangent 
vector $\dot{\gamma}(0)$ can be represented by $\gamma'(0)$ in the way that 
\be \label{eq:tvec} \dot{\gamma}(0)f = \nabla f(u)^T\gamma'(0). \ee
Since $\M$ is a level set of the constant-rank function \textbf{G}, the tangent space 
$T_u\M$ reads
\be \label{eq:tspace} T_u\M = \mathrm{ker}(\nabla \mathbf{G}(u)) 
= \{\xi \in\R^N\mid u^T \xi = 0, u^T \Gamma \xi = 0 \}. \ee
We naturally define the inner product$  $ $\langle \cdot, \cdot \rangle_u$ and the norm 
$\| \cdot \|_u$ on $T_u\M$ as
%\be \label{eq:prod}
% \langle \xi_u, \zeta_u \rangle_u := \xi_u^T \zeta_u, \quad
%\| \xi_u \|_u := \sqrt{\xi_u^T \xi_u}, 
%~~~ \xi_u, \zeta_u \in T_u\M. 
%\ee 
\be \label{eq:prod}
\langle \xi, \zeta \rangle_u := \xi^T \zeta, \quad \| \xi \|_u := \sqrt{\xi^T \xi}, 
~~~ \xi, \zeta \in T_u\M.
\ee 
Under such a metric, 
the Riemannian gradient grad$f(u)$, defined as the unique element of $T_u\M$ satisfying
%\be
%\langle \mathrm{grad}f(u), \xi_u \rangle_u = \xi_u f, ~~~\forall \xi_u \in T_u\M, 
%\ee
\be 
\langle \mathrm{grad}f(u), \xi \rangle_u = \xi f=\frac{d}{dt}f(u+t\xi)\big|_{t=0}, ~~~\forall\, \xi \in T_u\M,
\ee
can be written as 
\be \label{eq:grad2} \mathrm{grad}f(u) = \mathrm{P}_u\nabla f(u), \ee
where $\mathrm{P}_u$ denotes the orthogonal projection from $\R^N$ onto $T_u\M$.
From \eqref{eq:tspace} we can easily derive the formula of $\mathrm{P}_u$:

\begin{lemma}[$\mathrm{P}_u$] \label{lemma:px}
For an arbitrary point $w\in \R^N$, the orthogonal projection of it onto $T_u\M$ reads 
\be \label{Pz0} \mathrm{P}_u w = w - \frac{u^T \Gamma ^2 u \cdot u^T w - M u^T \Gamma w}{u^T \Gamma^2 u - M^2}u 
+ \frac{M u^T w - u^T \Gamma w}{u^T \Gamma^2 u - M^2}\Gamma u . \ee
\end{lemma} 
\begin{proof}
From \eqref{eq:tspace} and the definition of $\mathrm{P}_u$, we have 
\be (T_u\M)^{\perp} = \{ \alpha u + \beta \Gamma u \mid \alpha ,\beta \in \R \}, \ee
and $w - \mathrm{P}_u w \in (T_u\M)^{\perp}$, 
therefore there exists $\alpha _w ,\beta _w \in \R$ such that
\be \label{eq:Pz} \mathrm{P}_u w = w -  \alpha _w u - \beta _w \Gamma u. \ee
Noticing that $\mathrm{P}_u w \in T_u\M$ which implies
\bea \label{eq:TM} u^T (\mathrm{P}_u w) = 0, \quad u^T \Gamma (\mathrm{P}_u w) = 0. \eea
We can obtain from \eqref{eq:Pz}  that
\bea \label{eq:TM1} \alpha _w u^Tu + \beta _w u^T \Gamma u = u^T w, \quad \alpha _w u^T \Gamma u + \beta _w u^T \Gamma^2 u = u^T \Gamma w.\eea
%\\\label{eq:TM2} \alpha _w u^T Au + \beta _w u^T A^2 u &=& u^T Aw. \eea
In view of the fact that  $u^T u = 1$ and $u^T \Gamma u = M$,  \eqref{eq:TM1} %and \eqref{eq:TM2} 
can be
simplified as 
\be \label{eq:lsys} \left( \begin{matrix}
1 & M \\
M & u^T \Gamma^2 u 
\end{matrix} \right)
\left( \begin{matrix}
\alpha _w \\ \beta _w 
\end{matrix} \right)
 = 
 \left( \begin{matrix}
 u^T w \\ u^T \Gamma w
\end{matrix} \right). \ee
It follows directly from \textit{Cauchy-Schwarz inequality} that
\bea u^T \Gamma^2 u = \sum_{l=-F}^{F}l^2 \Vert u_l\Vert_2^2 
> \frac{(\sum_{l=-F}^{l=F} l\Vert u_l\Vert_2^2)^2}{\sum_{l=-F}^{l=F} \Vert u_l\Vert_2^2} 
= \frac{(u^T \Gamma u)^2}{u^T u} = M^2, \eea
which ensures the linear system \eqref{eq:lsys} has a unique solution:
\be \label{eq:ab}
\left( \begin{matrix}
\alpha _w \\ \beta _w
\end{matrix} \right)
= \frac{1}{u^T \Gamma^2 u-M^2}
 \left( \begin{matrix}
u^T \Gamma^2 u & - M \\ 
-M & 1 
\end{matrix} \right)
 \left( \begin{matrix}
 u^T w \\ u^T \Gamma w
\end{matrix} \right). \ee
Substituting \eqref{eq:ab} into \eqref{eq:Pz} yields the formula \eqref{Pz0}.
\end{proof}

Let $\mathfrak{X
}(\M)$ be the set of smooth vector fields on $\M$. The Riemannian 
Hessian Hess$f(u)$ is a linear mapping from $T_u\M$ into itself defined as
\be \label{eq:hess1} \mathrm{Hess}f(u)[\xi_u] = 
(\widetilde{\nabla}_{\xi}\mathrm{grad}f)(u), ~~~ \xi \in \mathfrak{X}(\M), \ee
where $\widetilde{\nabla}$ denotes the Riemannian connection of $\M$.
Since $\M$ is a Riemannian submanifold of $\R^N$, according to \cite{Absil} its Riemannian 
connection reads 
\be \label{eq:rcon} (\widetilde{\nabla}_{\xi}\eta)(u) 
= \mathrm{P}_u(\nabla\eta_u\cdot\xi_u), 
~~~ \xi ,\eta \in \mathfrak{X}(\M). \ee
Thus we have 
\be\label{eq:hess2}
\mathrm{Hess}f(u)[\xi] = \mathrm{P}_u(\nabla\mathrm{grad}f(u)\cdot\xi), 
~~~ \xi \in T_u\M. 
\ee
The formula of Hess$f(u)$ is given in following lemma: 
\begin{lemma}[$\mathrm{Hess} f(u)$] \label{lemma:hess}
Given a tangent vector $\xi \in T_u\M$, and let $g$ and $H$ be the Euclidean gradient 
and Euclidean Hessian of $f$ respectively, then 
\be \label{Hess0} \mathrm{Hess}f(u)[\xi] = h_e - \alpha _g \xi -\beta _g \Gamma \xi 
- \frac{u^T \Gamma^2 u\cdot u^T h_e -M \beta_u}{\alpha_u} u + \frac{M u^T h_e -\beta_u}{\alpha_u} \Gamma u, \ee
where 
\be \alpha _g := (1+\frac{M^2}{\alpha_u})u^T g-\frac{M}{\alpha_u}u^T \Gamma g, 
\quad \beta _g := -\frac{M}{\alpha_u}u^T g+\frac{1}{\alpha_u}u^T \Gamma g, \ee
and $h_e := H(u)\cdot \xi, \alpha_u := u^T \Gamma^2 u - M^2, \beta_u := u^T \Gamma h_e - \beta_g u^T \Gamma^2 \xi$.
\end{lemma}
\begin{proof}
Recalling Lemma \ref{lemma:px}  and \eqref{eq:grad2}, we get 
\be \mathrm{grad}f(u) = g(u) - \alpha _g u - \beta _g \Gamma u, \ee
and
\bea \nabla \mathrm{grad}f(u)\cdot \xi &=& \nabla g(u)\cdot \xi 
- \nabla(\alpha _g u)\cdot \xi - \nabla(\beta _g \Gamma u)\cdot \xi \nonumber\\
&=& h_e - \alpha _g \xi - \beta _g \Gamma \xi - (\nabla \alpha _g ^T\xi)u - (\nabla \beta _g ^T\xi)\Gamma u. \nonumber
\eea
Since $(\nabla \alpha _g ^T\xi)u , ~(\nabla \beta _g ^T\xi)\Gamma u \in (T_u\M)^{\perp}$,  we have
\be \label{eq:hess-l1} \mathrm{Hess}f(u)[\xi] 
= \mathrm{P}_u (\nabla \mathrm{grad}f(u)\cdot \xi) 
= \mathrm{P}_u(h_e -\alpha _g \xi -\beta _g \Gamma \xi). \ee
%And from $\xi \in T_u\M$, we have 
%\be \label{eq:hess-l2} u^T (h_e -\alpha _g \xi -\beta _g A\xi) = u^T h_e , \quad 
%u^T A(h_e -\alpha _g \xi -\beta _g A\xi) = b. \ee
For $\xi \in T_u\M$,  Lemma \ref{lemma:px}  and  \eqref{eq:hess-l1} lead to the formula \eqref{Hess0}.
\end{proof}

The first-order and second-order optimality conditions for optimization problems 
on Riemannian manifolds coincide
with the conventional ones \cite{Yang}.
If $u^{\ast}$ is a local solution of \eqref{dprob},  we have grad$\widetilde{E}(u^{\ast}) 
= 0$ and all the points $u$ at which grad$\widetilde{E}(u) = 0$ are called \textit{stationary points} of 
$\widetilde{E}$.

Line search optimization methods in $\R^N$ are based on the update formula
\be \label{eq:ls} u_{k+1} = u_{k} + t_k \eta_k, \ee
where $\eta_k \in \R^N$ is the \textit{search direction} and $t_k >0$ is the 
\textit{step size}. Correspondingly, when \eqref{eq:ls} is generalized to a manifold, 
$\eta _k$ is selected as a tangent vector, and the line search procedure relies on
the concept of \textit{retraction}:

\begin{definition}[retraction] \label{def:retr}
A retraction on a manifold $\M$ is a smooth mapping $R$ from the tangent bundle
$T\M$ onto $\M$ with the following properties. Let $R_u$ denote the restriction of 
$R$ to $T_u\M$.
\begin{itemize}
\item[(i)] $R_u(0_u) = u$, where $0_u$ denotes the zero element of $T_u\M$.
\item[(ii)] With the canonical identification $T_{0_u}T_u\M \simeq T_u\M$, $R_u$
satisfies
\be \label{eq:localrigidity} \mathrm{D}R_u(0_u) = \mathrm{id}_{T_u\M}, \ee
where $\mathrm{id}_{T_u\M}$ denotes the identity mapping on $T_u\M$.
\end{itemize}
\end{definition}

\begin{remark}
When $M\in \Z$, $\M$ is not a well-defined manifold. However, 
by restricting the feasible region to 
$$\widetilde{\M} := \left\{ u\in \M \mid \mathrm{at~ least~ two~ components~ of~} u
\mathrm{~is~nonzero}\right\}, $$
we can also define above structures and the formulas still work.
This modification does not change our numerical experiments.
\end{remark}

\section{A Modified Adaptive Regularized Newton Method}
\label{sect:arnt}

We aim to solve \eqref{dprob} with a modified version of the adaptive regularized 
Newton method (ARNT) developed in \cite{HU}. At the $k$-th iteration, ARNT uses a 
second-order Taylor model with a penalization term to approximate the 
original objective function but keeps the constraint $u \in \M$. 

Specifically, the method replaces \eqref{dprob} with a sequence of quadratic subproblems:
\be \label{prob-sub} \min_{u\in \M} m_k (u) := \langle \nabla \widetilde{E}(u_k), u-u_k\rangle + 
\frac{1}{2}\langle H_k (u-u_k),u-u_k\rangle + \frac{\sigma_k}{2}\| u-u_k \|_2^2 , \ee
where $\langle \cdot , \cdot \rangle$ denotes the dot product in $\R^N$ and $H_k$ is the 
Euclidean Hessian of $\widetilde{E}$ at $u_k$. The subproblem \eqref{prob-sub} is solved approximately
by applying a modified conjugate gradient (CG) method to the linear system
\be \label{prob-lsub} \mathrm{grad}\, m_k (u_k) + \mathrm{Hess}\, m_k (u_k)[\xi] = 0. \ee
The method terminates when either certain accuracy is reached or negative curvature
is detected. It outputs two vectors $s_k$ and $d_k$, where $s_k$ is the solution computed by CG method and $d_k$ represents 
the negative curvature information. The new search direction $\xi_k$ is chosen as
\begin{equation} \label{eq:xi}
\xi_k =\begin{cases}
s_k + \tau_k d_k & {\mathrm{if~} d_k \neq 0}, \\
s_k & {\mathrm{if~} d_k = 0}, 
\end{cases}
\quad \mathrm{with~} \tau_k := \frac{\langle d_k ,\mathrm{grad}\, m_k (u_k)\rangle_{u_k}}
{\langle d_k ,\mathrm{Hess}\, m_k (u_k)[d_k]\rangle_{u_k}}, 
\end{equation}
which is  a descent direction  (cf. Lemma 7, \cite{HU}). 

After construction of $\xi_k$, a monotone Armijo-based curvilinear search 
is conducted to generate a trial point
\be \label{eq:tp1} z_k = R_{u_k}(\alpha_0 \delta^\varsigma \xi_k), \ee
where $\varsigma$ is the smallest integer satisfying 
\be \label{eq:tp2} m_k (R_{u_k} (\alpha_0 \delta^\varsigma \xi_k)) \leq \rho \alpha_0 
\delta^\varsigma 
\langle \mathrm{grad}\, m_k (u_k), \xi_k \rangle_{u_k} \ee
and $\rho ,\delta \in (0,1), \alpha_0 \in (0,1]$ are given constants.

In order to monitor the acceptance of the trial point $z_k$ and adjust the regularization 
parameter $\sigma_k$, the above procedure is embedded in a trust-region framework 
where $\sigma_k$ plays a similar role as the trust-region radius and is updated according 
to the ratio 
\be \label{eq:ratio} \rho_k = \frac{\widetilde{E}(z_k)-\widetilde{E}(u_k)}{m_k (z_k)}. \ee

ARNT may exhibit a certain instability when directly applied to solve \eqref{dprob}.
To improve its performance, we combine it with the cascadic multigrid method in
\cite{Bornemann}.  In detail, we first solve \eqref{dprob} on the coarsest mesh, and then 
use the obtained solution as the initial guess of the problem on a finer mesh, and 
repeat until reaching the finest mesh. 

On each mesh, we use the Riemannian gradient method with  a BB step size (RGBB) in \cite{HU} 
to compute an initial point for ARNT. At the $k$-th iteration, RGBB performs a nonmonotone 
Armijo-based curvilinear search along the steepest descent direction 
$\eta_k = -\grad \widetilde{E}(u_k)$.
Given $\rho , \varrho , \delta \in (0,1)$, it tries to find the smallest integer $\varsigma$ satisfying 
\be \label{NMN} \widetilde{E}(R_{u_k}( \gamma_k \delta^\varsigma \eta_k)) \leq C_k +\rho  \gamma_k \delta^\varsigma
 \langle \mathrm{grad}\widetilde{E}(u_k),\eta_k\rangle_{u_k}, \ee
where the initial step size $\gamma_k$ is computed as 
$\gamma_k = |\langle s_{k-1},v_{k-1}\rangle_{u_k}| / \langle v_{k-1}, v_{k-1}\rangle_{u_k}$,
with 
$ s_{k-1} = u_k - u_{k-1}, \quad v_{k-1} = \mathrm{grad}\widetilde{E}(u_k) - \mathrm{grad}\widetilde{E}(u_{k-1}). $
The value $C_{k+1}$ is calculated via $C_{k+1} = (\varrho Q_k C_k + \widetilde{E}(u_{k+1}))/Q_{k+1}$, 
with $C_0 = \widetilde{E}(u_0), Q_{k+1} = \varrho Q_k + 1$ and $Q_0 = 1$. 

The modified adaptive regularized Newton method (still referred to as ARNT in this paper) 
is presented in Algorithm \ref{alg:ARNT}.

\begin{algorithm2e}[t]\caption{A Modified Adaptive Regularized Newton Method}
\label{alg:ARNT}%\LinesNumberedHidden
Choose an initial mesh $\mathcal{T}^0$ and $u^{(0)}$. Set $j=0$. \\

\While{$j \leq m$}{
Input $u_0 = u^{(j)}$. Set $k=0, C_0=\widetilde{E}(u_0), Q_0=1$. \\
\While{stopping conditions not met}{
Compute $\eta_k = -\grad \widetilde{E}(u_k)$. \\
Compute $\gamma_k, C_k, Q_k$ and find the $\varsigma$ satisfying \eqref{NMN}. \\
Set $u_{k+1} \leftarrow R_{u_k}(\gamma_k \delta^\varsigma \eta_k)$. \\
Set $k\leftarrow k+1$.}
\smallskip

Input $u_0=u_k$. Choose $0 < \eta_1 \leq \eta_2 < 1, 0 < \gamma_0 < 1 < \gamma_1 
\leq \gamma_2$ and an initial regularization parameter $\sigma_0 > 0$. Set $k=0$. \\
\While{stopping conditions not met}{
Compute a new trial point $z_k$ according to \eqref{eq:tp1} and \eqref{eq:tp2}. \\
Compute the ratio via \eqref{eq:ratio}. \\
\If{$\rho_k \geq \eta_1$}{
Set $u_{k+1} = z_k$. \\
\lIf{$\rho_k \geq \eta_2$}{choose $\sigma_{k+1} \in (0,\gamma_0 \sigma_k]$}
\lElse{choose $\sigma_{k+1} \in [\gamma_0 \sigma_k, \gamma_1 \sigma_k]$}
}
\Else{
Set $u_{k+1} = u_k$. \\
Choose $\sigma_{k+1} \in [\gamma_1 \sigma_k, \gamma_2 \sigma_k]$.
}
$k \leftarrow k+1$.}
\smallskip

Set $u^{(j+1)}=u_k$. Refine the mesh $\mathcal{T}^j$ uniformly to obtain $\mathcal{T}^{j+1}$.\\
$j \leftarrow j+1$.
}
\end{algorithm2e}

%\begin{algorithm}
%\caption{Build tree}
%\label{alg:buildtree}
%\begin{algorithmic}
%\STATE{Choose an initial mesh $\mathcal{T}^0$ and $u^{(0)}$. Set $j=0$.}
%\WHILE{$j \leq m$}
%	\STATE{Input $u_0 = u^{(j)}$. Set $k=0, C_0=\widetilde{E}(u_0), Q_0=1$.}
%	\WHILE{stopping conditions not met}
%		\STATE{Compute $\eta_k = -\grad \widetilde{E}(u_k)$.}
%		\STATE{Compute $\gamma_k, C_k, Q_k$ and find the $\varsigma$ 
%		satisfying \eqref{NMN}.}
%		\STATE{Set $u_{k+1} \leftarrow R_{u_k}(\gamma_k \delta^\varsigma 
%		\eta_k)$.}
%		\STATE{Set $k\leftarrow k+1$.}
%	\ENDWHILE
%	\smallskip
%	
%	\STATE{Input $u_0=u_k$. Choose $0 < \eta_1 \leq \eta_2 < 1, 0 < \gamma_0 
%	< 1 < \gamma_1 \leq \gamma_2$ and an initial regularization parameter 
%	$\sigma_0 > 0$. Set $k=0$.}
%	\WHILE{stopping conditions not met}
%		\STATE{Compute a new trial point $z_k$ according to \eqref{eq:tp1} and 
%		\eqref{eq:tp2}.}
%		\STATE{Compute the ratio via \eqref{eq:ratio}.}
%		\IF{$\rho_k \geq \eta_2$}
%			\STATE{Set $u_{k+1} = z_k$.}
%			\IF{$\rho_k \geq \eta_2$}
%				\STATE{choose $\sigma_{k+1} \in (0,\gamma_0 \sigma_k]$}
%			\ELSE{choose $\sigma_{k+1} \in [\gamma_0 \sigma_k, 
%			\gamma_1 \sigma_k]$}
%			\ENDIF
%		\ELSE{Set $u_{k+1} = u_k$.
%		Choose $\sigma_{k+1} \in [\gamma_1 \sigma_k, \gamma_2 \sigma_k]$.}
%		\ENDIF
%		\STATE{$k \leftarrow k+1$.}
%	\ENDWHILE
%	\smallskip
%	
%	\STATE{Set $u^{(j+1)}=u_k$. Refine the mesh $\mathcal{T}^j$ uniformly to 
%	obtain $\mathcal{T}^{j+1}$.}
%	\STATE{$j \leftarrow j+1$.}
%\ENDWHILE
%\end{algorithmic}
%\end{algorithm}

\section{Retractions}
\label{sect:retr}

The selection of retractions can affect the performance of Riemannian optimization algorithms. 
In this section, we try to find retractions of the form 
\be \label{eq:oretr0} R_u (\xi_u) := \psi (u+\xi_u), \quad u\in \M, \xi_u \in T_u\M, \ee
% by looking at the calculation of $R_x (t\xi)$ for 
%given $x\in \M$, $\xi \in T_x\M$ and a sufficiently small 
%$t\in \R^{+}$. In particular, we will discuss retractions with the form 
where $\psi $ is some ``projection" from a neighborhood of $\M$ in $\R^N$ to $\M$.

For $w=(w_F,w_{F-1},\ldots,w_{-F})\in \R^N$ ($w_l\in\R^{2n}$, $N=2n(2F+1)$, $l=F,\ldots,-F$), we define two functions $f_1(w),f_2(w)\in\{F,F-1,\ldots,-F\}$: 
\be f_1 (w) := \min \{ l \mid w_l \neq 0 \},\quad f_2 (w) := \max \{ l\mid w_l\neq 0\} , 
\quad \forall w \in \R^N .\ee
Observe that at every point $u\in \M$, the constraints indicate 
\be \sum_{l=-F}^F (l-M) \|u_l\|_2^2 = 0. \ee
%Denote $w := x + t\xi$. Observe that the constraints
%\be \label{eq:constaint} \sum_{l=-F}^F \| x_l \|_2^2 = 1, 
%~~~ \sum_{l=-F}^F l\| x_l \|_2^2 = M \ee
%and $M \notin \Z$ indicate 
Thus when $M\notin \Z$, we have $f_1(u) < M, f_2(u) > M$,
which is equivalent to 
\be \label{eq:noneq6} \sum_{l<M}\| u_l \|_2^2 > 0, ~~~ \sum_{l>M}\| u_l \|_2^2 > 0; \ee
when $M\in \Z$, \eqref{eq:noneq6} also holds for $u\in \widetilde{\M}$.
Define the open set $\Omega$ as 
\be\Omega := \left\{ w\in \R^N \mid \sum_{l<M}\| w_l \|_2^2 > 0, ~\sum_{l>M}\| w_l \|_2^2 > 0  
\right\} ,\ee
then $\Omega\supset \M$ for $M\notin \Z$ and $\Omega\supset \widetilde{\M}$ for 
$M\in \Z$.
For the simplicity of presentation, we will discuss three different retractions from $\Omega$ to $\M$ ($M\notin \Z$) and all the results hold also for $\widetilde{\M}$ ($M\in Z$). 

\subsection{Projective Retraction}
\label{sect:proj}

The most intuitive retraction $\psi$ is given by the \textit{projection operator} 
$\cP_{\M}$, which is defined as 
\be \label{prob-sub2} \cP_{\M}(w)=\arg \min_{z \in \M}\quad \frac{1}{2} \|z-w\|_2^2 , 
~~~ w \in \R^N. \ee
According to \cite{AbsilR}, $\cP_{\M}$ is a well-defined function (existence and uniqueness of the projection hold) in a neighborhood 
$\widetilde{\Omega}\subset \Omega$ of $\M$, 
and the mapping $R_u(\xi_u):= \cP_{\M}(u+\xi_u)$ is a well-defined retraction on $\M$, called 
the \textit{projective retraction} in this paper.
%and the function $R: T\M \rightarrow \M$ defined by 
%\be R_u :T_u\M \longrightarrow \M : \xi_u \longmapsto \cP_{\M} (u+\xi_u), \quad u \in \M \ee
%is a retraction, which is called the projective retraction in this paper.
%Thus for a sufficiently small $t$, 
%the point $w$ has a unique projection $\cP_{\M}(w)$. 
The explicit formula is given as follows.
\begin{lemma} \label{lemma:r1}
For an arbitrary point $w \in \widetilde{\Omega}$, $\cP_{\M}  (w)$ reads 
\be \cP_{\M}  (w)_l=\begin{cases}
0 & {\mathrm{if~} l<f_1(w)\mathrm{~~or~~} l>f_2(w)}, \\
{\frac{w_l}{1-\mu -l\lambda}} & {\mathrm{if~} f_1(w)\leq l \leq f_2(w)},
\end{cases}
\ee
where
\bea \lambda = r\sqrt{\sum_{l=f_1(w)}^{f_2(w)}\frac{\| w_l \|_2^2}{[1-(l-M)r]^2}}~, \quad
\mu = 1 - (1+Mr)\sqrt{\sum_{l=f_1(w)}^{f_2(w)}\frac{\| w_l \|_2^2}{[1-(l-M)r]^2}}~, \eea
and $r$ is the unique zero of the function 
\be h_1(t) = \sum_{l=f_1(w)}^{f_2(w)}\frac{(l-M)\| w_l \|_2^2}{[1-(l-M)t]^2}, 
\quad t \in ( \frac{1}{f_1(w)-M}, \frac{1}{f_2(w)-M} ).\label{eq:h1} \ee
\end{lemma} 
\begin{proof}
Define the Lagrangian function of \eqref{prob-sub2} as
\be L(z,\mu, \lambda) =  \frac{1}{2} \|z-w\|_2^2 - \frac{\mu}{2} \left( \|z\|_2^2 -1\right) -
\frac{\lambda}{2} \left( \sum_{l=-F}^F l
\|z_l\|_2^2 - M\right).\ee
Let $z=\cP_{\M}(w)$, then the condition $\nabla_l L(z,\mu, \lambda) = 0$ gives
  \bea \label{eq:optw1}
(1-\mu-l\lambda)z_l=w_l, \quad l=F, ..., -F. \eea
Since $\|z-w_l\|_2^2=\|z\|_2^2+\|w\|_2^2-2 \sum_l w_l^Tz_l=1+\|w\|_2^2-2 \sum_l w_l^Tz_l$ ($z\in\M$), the projection $z$ maximizes $w_l^Tz_l$ which leads to $1-\mu-l\lambda>0$ for $w_l\neq0$. 
On the other hand, if $w_l=0$ then $z_l=0$, otherwise substituting $z_l$ with$-z_l$ yields a different projection of $w$, which contradicts the uniqueness.

Since $w_{f_1(w)}$ and $w_{f_2(w)}$ are nonzero, we have 
\be \label{lcond1} 1-\mu -f_1(w) \lambda > 0, \quad 1-\mu -f_2(w) \lambda > 0, \ee
which is equivalent to  $1-\mu -M\lambda > (f_1(w)-M)\lambda$ ,  $1-\mu -M\lambda > (f_2(w) -M)\lambda$ , 
and 
\bea \label{lcond2} 1-\mu -M\lambda > 0, \quad 
 \frac{1}{f_1(w)-M} < \frac{\lambda}{1-\mu -M\lambda} < \frac{1}{f_2(w) -M}. \eea
The inequalities in \eqref{lcond1}  indicate that $1-\mu -l\lambda>0$  for 
%$l\in\{l \in \Z\mid f_1(w)\leq l\leq f_2(w)\}$, 
$l = f_1(w)+1, ..., f_2(w)-1,$
and from \eqref{eq:optw1} and $z \in \M$ we have
\bea \label{lcond3}\sum_{l=f_1(w)}^{f_2(w)} \frac{\| w_l\|_2^2}{(1-\mu -l\lambda)^2} = 1, 
\quad \sum_{l=f_1(w)}^{f_2(w)} \frac{l\| w_l\|_2^2}{(1-\mu -l\lambda)^2} = M. \eea
In view of \eqref{lcond2} - \eqref{lcond3}, denoting
\be \label{eq:st} s = (1-\mu -M\lambda)^2, \quad
r = \frac{\lambda}{1-\mu -M\lambda}, \ee
recalling the definition of function $h_1(\cdot)$ in \eqref{eq:h1}, we have  $ s>0, \quad \frac{1}{f_1(w)-M}<r<\frac{1}{f_2(w)-M} $ and 
\be \label{eq:rprob} h_1(r) = 0, \quad 
s = \sum_{l=f_1(w)}^{f_2(w)}\frac{\| w_l \|_2^2}{[1-(l-M)r]^2}. \ee
For any $t \in ( \frac{1}{f_1(w)-M}, \frac{1}{f_2(w)-M} )$, we have
\be \label{hderi} h_1'(t) = \sum_{l=f_1(w)}^{f_2(w)} \frac{2(l-M)^2\| w_l \|_2^2}
{[1-(l-M)t]^3}~ > 0. \ee
In addition, noticing that
\be \label{hlim} \lim_{t\rightarrow \frac{1}{f_1(w)-M}+0}h_1(t) = -\infty, 
 \lim_{t\rightarrow \frac{1}{f_2(w)-M}-0}h_1(t) = +\infty, \ee
 $h_1(\cdot)$ has exactly one zero in $( \frac{1}{f_1(w)-M}, \frac{1}{f_2(w)-M} )$.
Substituting \eqref{eq:rprob} into \eqref{eq:st},  the formulas of $\lambda$ and $\mu$ 
can be obtained accordingly.
\end{proof}

We remark that \eqref{lemma:r1} can be applied to any $w\in \Omega$.
For spin-1 case, the closed-form solution of \eqref{prob-sub2} is computable.
\begin{lemma}\label{lemma:subu} 
When $F=1$, given any nonzero $w\in \Omega$, the optimal solution 
$z=(z_1;z_0;z_{-1})$ of \eqref{prob-sub2} is

(1) If $M=0$, then $z_0 = w_0/t$ and 
\be z_l = \frac{\|w_1\|_2+\|w_{-1}\|_2}{2 t \|w_l\|_2} w_l, \quad l = \pm 1, \ee
with $t = \sqrt{\|w_0\|_2^2 + (\|w_1\|_2+\|w_{-1}\|_2)^2/2}$.

(2) If $M>0$ and $w_0=0$, then $z_0=0$, 
\be z_l=\frac{\sqrt{1+lM}}{\sqrt{2}\|w_l\|_2}w_l,  \quad l=\pm1. \ee

(3) If $M>0$ and $w_{-1}=0$, then $z_{-1}=0$, 
\be z_0 = \frac{\sqrt{1-M}}{\|w_0\|_2} w_0, \quad z_1 = \frac{\sqrt{M}}{\|w_1\|_2} w_1. \ee

(4) If $M>0$ and $\|w_0\|, \|w_{-1}\|>0$, then $z_l=w_l/(1-\mu-l\lambda)$ ($l=\pm1,0$) 
with 
\begin{equation}
\mu=1-\frac{\|w_0\|_2}{\alpha},\qquad \lambda=\frac{\|w_1\|_2}{\sqrt{1+M-\alpha^2}}-\frac{\|w_0\|_2}{\alpha},
\end{equation}
where
\be \alpha=\frac{\sqrt{1-M^2}\beta}{\sqrt{2M+(1+M)\beta^2}}, \ee
with $\beta$ depending on $\|w_l\|$ as
\begin{align}
&\beta=\frac{\xi}{2}+S-\frac{1}{2}\sqrt{-4S^2-2p_0-q_0/S},\qquad p_0=\frac{-\xi^2-2\zeta^2+2}{2},\quad q_0=-\xi(\zeta^2+1),\\
&\Delta_0=(\xi^2-\zeta^2+1)^2,\quad \Delta_1=2(\xi^2-\zeta^2+1)^3+108\xi^2\zeta^2,
\end{align}
and
\be
S=\begin{cases}
\frac{1}{2}\sqrt{-\frac{2}{3}p_0+\frac{1}{3}\left(Q+\frac{\Delta_0}{Q}\right)},\qquad\text{with } Q=\sqrt[^3]{\frac{\Delta_1+\sqrt{\Delta_1^2-4\Delta_0^3}}{2}},
&x^{2/3}+1\leq y^{2/3},\\
\frac{1}{2}\sqrt{-\frac{2}{3}p_0+\frac{2}{3}\sqrt{\Delta_0}\cos\left(\frac{1}{3}\arccos\left(\frac{\Delta_1}{2\sqrt{\Delta_0^3}}\right)\right)},&
x^{2/3}+1> y^{2/3},
\end{cases}
\ee
\be \xi= \frac{2\sqrt{M}\|w_0\|_2}{\sqrt{1+M}\|w_{-1}\|_2},
\quad \zeta= \frac{\sqrt{1-M}\|w_1\|_2}{\sqrt{1+M}\|w_{-1}\|_2}.\ee
\end{lemma}
\begin{proof}
The first three cases are straightforward to verify. Here we only present the proof for 
case (4).
If $M>0$ and $\|w_0\|_2, \|w_{-1}\|_2>0$, then $z_l=w_l/(1-\mu-l\lambda)$ ($l=0,\pm1$). 
Let $\|z_0\|_2=s$, we have from \eqref{eq:optw1} that $1-\mu=\|w_0\|_2/s$ and
\be
\frac{\sqrt{2}\|w_1\|_2}{\sqrt{1+M-s^2}}+\frac{\sqrt{2}\|w_{-1}\|_2}{\sqrt{1-M-s^2}}=2\frac{\|w_0\|_2}{s},
\ee
which implies
\be\label{eq:sol4}
\frac{s}{\sqrt{2}\sqrt{1+M-s^2}}\cdot\frac{\|w_1\|_2}{\|w_0\|_2}+\frac{s}{\sqrt{2}\sqrt{1-M-s^2}}\cdot\frac{\|w_{-1}\|_2}{\|w_0\|_2}=1.
\ee
There exists a unique solution $s\in(0,\sqrt{1-M})$ and the Lagrange multipliers can be identified.
\end{proof}

\subsection{Orthogonal Retraction}
\label{sect:orth}

Inspired by the projective retraction, we can consider $\psi$ of the form 
\be \label{eq:oretr1} \psi  (w)_l = \sigma_l w_l , \quad l = F, ..., -F \ee
with undetermined positive coefficients $\sigma_F , ..., \sigma_{-F}$.
Besides the constraints 
\be \label{eq:oretr2} \sum_{l=-F}^{F} \|w_l\|_2^2 \sigma_l^2 = 1, 
\quad \sum_{l=-F}^{F} l\|w_l\|_2^2 \sigma_l^2 = M, \ee
we  have to introduce additional $2F-1$ conditions to uniquely determine the $2F+1$ coefficients.

In \cite{Bao1}, Bao and Lim proposed the condition $\sigma_1 \sigma_{-1} = 
\sigma_0 ^2$ for spin-1 BEC. It can be generalized to 
\begin{equation} \label{eq:orth}
\begin{cases}
\sigma_{l-1} \sigma_{l+1} = \sigma_l^2 , \quad  &l = 1, 2, ..., F-1, \\
 \sigma_l \sigma_{-l} = \sigma_0^2, \quad &l = 1, 2, ..., F 
\end{cases}
\end{equation}
for spin-$F$ cases. The mapping $R$ characterized by \eqref{eq:oretr0} and 
\eqref{eq:oretr1} -  \eqref{eq:orth} is called the \textit{orthogonal retraction} in this paper.
\begin{lemma} \label{lemma:r2}
For an arbitrary point $w\in \Omega$, $\psi(w)$ defined by \eqref{eq:oretr1} -  
\eqref{eq:orth} reads 
%\be \psi(w)_l = \sqrt{\frac{r^l}{\sum_{k=f_1(w)}^{f_2(w)}\|w_l\|_2^2 r^k}}\cdot w_l ,\quad 
%l = F, ..., -F, \ee
\be \psi(w)_l = \sqrt{\frac{r^l}{\sum_{k=-F}^{F}\|w_l\|_2^2 r^k}}\cdot w_l ,\quad 
l = F, ..., -F, \ee
where $r$ is the unique positive zero of the polynomial
\be h_2(t) = \sum_{l=0}^{2F} (l-F-M)\|w_{l-F}\|_2^2 \cdot t^l. \ee
\end{lemma} 
\begin{proof}
%From 
%\be \sum_{l=0}^F \|u_l\|_2^2 ~\geq ~ \sum_{l>M} \|u_l\|_2^2 ~> ~0 \ee
%and \eqref{eq:orth} we know $\sigma_0 > 0$, which in turn indicates that $\sigma_l > 0$ 
%for all $l$. 
Let $r = \sigma_1^2 / \sigma_0^2 > 0$, then
\be \label{eq:orth-2} \sigma_l = \sigma_0 r^{l/2}, \quad l = F, ..., -F. \ee
Substituting \eqref{eq:orth-2} into \eqref{eq:oretr2} yields 
%\be \label{eq:orth-3} \sum_{l=f_1(w)}^{f_2(w)}(l-M)\|w_l\|_2^2 r^l = 0, \quad 
%\sigma_0^2 = \frac{1}{\sum_{l=f_1(w)}^{f_2(w)}\|w_l\|_2^2 r^l}. \ee
\be \label{eq:orth-3} h_2(r)=0, \quad 
\sigma_0^2 = \frac{1}{\sum_{l=-F}^{F}\|w_l\|_2^2 r^l}. \ee
%Define function $h(r)$ as 
%\be h(r) = \sum_{l=0}^{f_2(w)-f_1(w)} (l+f_1(w)-M)\|w_{l+f_1(w)}\|_2^2 \cdot r^l, 
%\quad r \geq 0. \ee
%The first equation in \eqref{eq:orth-3} is equivalent to $h_2(r)=0$.
Introducing  $h_3(t)=h_2(t) t^{-f_1(w)-F}$, we have
%And positive roots of $h_2$ are the same as positive roots of the function
\be h_3(t) := \sum_{l=0}^{f_2(w)-f_1(w)} (l+f_1(w)-M)\|w_{l+f_1(w)}\|_2^2 \cdot t^l,\ee
and
\be h_3(0) = (f_1(w)- M)\|w_{f_1(w)}\|_2^2 < 0, 
\quad \lim_{t\rightarrow +\infty}h_3(t) = +\infty, \ee
which implies the function $h_3(\cdot)$ has at least one positive zero. 

Let $\displaystyle \sum_1$ and $\displaystyle \sum_2$ denote 
$\displaystyle \sum_{1\leq l<M-f_1(w)}$ and 
$\displaystyle \sum_{M-f_1(w)<l\leq f_2(w)-f_1(w)}$ 
respectively. At a zero $r_0$ of $h_3$, we have  
\bea &\sum_1 & l(l+f_1(w)-M)\|w_{l+f_1(w)}\|_2^2 r_0^{l-1}\nonumber \\
&\geq & (M-f_1(w))\sum_1 (l+f_1(w)-M)\|w_{l+f_1(w)}\|_2^2 r_0^{l-1} \nonumber\\
&=& \frac{M-f_1(w)}{r_0}\left[\sum_1 (l+f_1(w)-M)\|w_{l+f_1(w)}\|_2^2 r_0^l - h_3(r_0)\right] \nonumber\\
&=& (M-f_1(w))\left[\frac{(M-f_1(w))\|w_{f_1(w)}\|_2^2}{r_0}-
\sum_2 (l+f_1(w)-M)\|w_{l+f_1(w)}\|_2^2 r_0^{l-1}\right], \nonumber\eea
which leads to 
\bea h_3'(r_0) &=& \sum_1 l(l+f_1(w)-M)\|w_{l+f_1(w)}\|_2^2 r_0^{l-1} 
+ \sum_2 l(l+f_1(w)-M)\|w_{l+f_1(w)}\|_2^2 r_0^{l-1}\nonumber \\
&\geq & \label{orth-h}
\frac{(M-f_1(w))^2\|w_{f_1(w)}\|_2^2}{r_0} + \sum_2 (l+f_1(w)-M)^2\|w_{l+f_1(w)}\|_2^2 r_0^{l-1} 
> 0. \eea
From \eqref{orth-h} we can see that $h_3$ has exactly one positive zero, and
$h_2$  has exactly one positive zero too.
Substituting \eqref{eq:orth-3} into \eqref{eq:orth-2} leads to the formulas of the coefficients. 
\end{proof}

Noticing that in spin-1 cases, $h_2$ degenerates to a quadratic polynomial, and 
the orthogonal retraction has a closed form solution.

The well-definedness of the orthogonal retraction is guaranteed by following theorem
\cite{Absil}:
\begin{theorem}\label{theorem:retr}
Let $\M$ be an embedded manifold of a vector space $\mathcal{E}$ and let $\mathcal{N}$ 
be an 
abstract manifold such that $\dim(\M)+\dim(\mathcal{N})=\dim(\mathcal{E})$. Assume that 
there is a diffeomorphism 
$$\varphi : \M \times \mathcal{N}\rightarrow \mathcal{E}_{\ast}: (u,v)\mapsto \varphi (u,v),$$
where $\mathcal{E}_{\ast}$ is an open subset of $\mathcal{E}$, with a neural element $e$
satisfying 
$$\varphi (u,e) = u, \quad \forall u \in \M.$$
Then the mapping 
$$R_u (\xi_u) := \pi_1 (\varphi^{-1} (u+\xi_u)),$$
where $\pi_1 :\M\times \mathcal{N}\rightarrow \M$ is the projection onto the 
first component, defines a retraction on $\M$.
\end{theorem}

\begin{lemma}
The orthogonal retraction is a well-defined retraction on $\M$.
\end{lemma} 
\begin{proof}
Take $\mathcal{N}=\R_{+}^2$,  and $\mathcal{N}$ is a manifold satisfying  
$\dim(\M)+\dim(\mathcal{N})=\dim(\R^N)$. Define the mapping 
$\varphi : \M \times \mathcal{N}\rightarrow \Omega$ as 
\be \varphi(u, v) := \left( \frac{1}{v_1 v_2^{l/2}} \cdot u_l \right)_{l=-F}^F , \quad \forall (u,v)\in\M\times
\mathcal{N}. \ee
Lemma \ref{lemma:r2} shows 
for any $w\in\Omega$ there exists a unique $u=\psi (w), v = (\sigma_0, r)^T$ such that 
$\varphi(u,v)=w$, thus $\varphi$ is a bijection. It is obvious to see that $\varphi$ is smooth 
on $\M\times\mathcal{N}$, and $\varphi (u,\mathbf{1})=u, ~\forall u\in\M$. 

From Lemma \ref{lemma:r2}, we have 
\be \varphi^{-1}(w) = \left( \left( \sigma_0 r^{l/2}w_l \right)_{l=-F}^F, ~(\sigma_0, r)^T\right), \quad \forall w\in \Omega, \ee
where $\sigma_0 = \sqrt{\frac{1}{\sum_{l=-F}^{F}\|w_l\|_2^2 r^l}}$ and $r$ is characterized by 
the equation
\be h_2(r) = \sum_{l=0}^{2F} (l-F-M)\|w_{l-F}\|_2^2 \cdot r^l = 0. \ee 
Since $h_2'(r)=(h_3(t)\cdot t^{f_1(w)+F})'|_{t=r}=h_3'(r)\cdot r^{f_1(w)+F} > 0$, 
it follows from the implicit function theorem that $r$, when considered as a function of $w$,
is smooth. Then $\varphi^{-1}$ is also a smooth function at every $w\in\Omega$, which makes
 $\varphi$ a diffeomorphism. Thus the orthogonal retraction, given by 
 \be R_u (\xi_u) := \pi_1 (\varphi^{-1} (u+\xi_u)), \ee
 is a retraction on $\M$.
\end{proof}

\subsection{Closed-form Retraction}
\label{sect:close}

In the projective retraction, the coefficients take  the form 
\be \sigma_l = \frac{1}{1-\mu -l\lambda}. \ee
When $w \rightarrow \M$, we have $\sigma_l \rightarrow 1$ and $\mu , \lambda \rightarrow 0$,
and
\be \frac{1}{1-\mu -l\lambda} = \sqrt{1+2\mu +2l\lambda} + \mathrm{o}(\mu +l\lambda). \ee
Thus we can approximate the projective retraction by taking 
\be \label{eq:retr3} \sigma_l \approx\sqrt{1+2\mu +2l\lambda}, \quad l = F, ..., -F. \ee
As shown below,  \eqref{eq:retr3} has a closed-form formula, and the mapping $R$ characterized by \eqref{eq:oretr0}, 
\eqref{eq:oretr1}, \eqref{eq:oretr2} and \eqref{eq:retr3} is called the 
\textit{closed-form retraction} in this paper.
Firstly, we introduce 
\be 
S := \left\{ w\in \R^N \mid w^T \Gamma^2 w-(M+l)w^T \Gamma w + M l \cdot w^T w > 0, l = \pm F 
\right\}. 
\ee
Apparently, $S$ is an open set and $\M \subset S \subset \Omega$. We next discuss the 
computation of $\psi (w)$ for $w\in S$.
\begin{lemma} \label{lemma:r3}
For an arbitrary point $w\in S$, $\psi(w)$ defined by \eqref{eq:oretr1}, 
\eqref{eq:oretr2} and \eqref{eq:retr3} reads 
%\be \psi(w)_l = \sqrt{\frac{r^l}{\sum_{k=f_1(w)}^{f_2(w)}\|w_l\|_2^2 r^k}}\cdot w_l ,\quad 
%l = F, ..., -F, \ee
\be \psi(w)_l = \sqrt{\frac{w^T \Gamma^2 w-(M+l)w^T \Gamma w + M l \cdot w^T w}
{w^T w\cdot w^T \Gamma^2 w - (w^T \Gamma w)^2}}\cdot w_l ,\quad 
l = F, ..., -F. \label{eq:retract3}\ee
\end{lemma} 
\begin{proof}
%Define the set $S$ as 
%\be S := \left\{ w\in \R^N \mid w^T A^2 w-(M+l)w^T Aw + M l \cdot w^T w > 0, l = \pm F 
%\right\}. \ee
%Apparently, $S$ is an open set and $\M \subset S \subset \Omega$. We will discuss the 
%computation of $\psi (w)$ for $w\in S$.
Substituting $\eqref{eq:retr3}$ into $\eqref{eq:oretr2}$ yields 
\begin{equation}
\begin{cases}
(1 + 2\mu)\, w^T w+ 2\lambda w^T \Gamma w=1, \\
(1 + 2\mu)\, w^T \Gamma w + 2\lambda w^T \Gamma^2 w=M,
\end{cases}
\end{equation}
and the solution is given by
\be \label{eq:r3-2} \mu = \frac{1}{2}\cdot \frac{w^T \Gamma^2 w-M w^T \Gamma w}{w^T w\cdot w^T \Gamma^2 w 
- (w^T \Gamma w)^2}-\frac{1}{2} , \quad
\lambda = \frac{1}{2}\cdot \frac{-w^T \Gamma w+M w^T w}{w^T w\cdot w^T \Gamma^2 w - (w^T \Gamma w)^2}. \ee
The condition $w\in S$ ensures $1+2\mu + 2l \lambda > 0$ for $l=F, ..., -F$.  In view of the retraction \eqref{eq:retr3},
we obtain the formula \eqref{eq:retract3}.
\end{proof}

\begin{lemma} \label{lemma:r3-2}
The closed-form retraction is a well-defined retraction.
\end{lemma} 
\begin{proof}
Denote \be \mathcal{N} = \left\{ v=(v_1,v_2)^T\in \R^2 \mid 1+2v_1 + 2F v_2 > 0, ~
1+2v_1 - 2F v_2 > 0 \right\}. \ee
$\mathcal{N}$ is an open subset of $\R^2$ and therefore a 2-dimensional manifold.
Define the mapping $\varphi : \M \times \mathcal{N} \rightarrow \R^N$ as 
\be \varphi (u,v) := \left(  \frac{1}{\sqrt{1+2v_1 +2l v_2}}\cdot u_l \right)_{l=-F}^F, 
\quad \forall (u,v) \in \M \times \mathcal{N}. \ee

For an arbitrary point $(u,v)\in \M \times \mathcal{N}$, let $w=\varphi (u,v)$.
\begin{itemize}

\item If $v_2 = 0$, then $w = \frac{1}{\sqrt{1+2v_1}} u$ ($u\in\M$), and 
\bea && w^T \Gamma^2 w-(M+l)w^T \Gamma w + M l \cdot w^T w\nonumber \\
=&& \frac{1}{1+2v_1} [u^T \Gamma^2 u-(M+l)u^T \Gamma u + M l \cdot u^T u\nonumber\\
=&& \frac{1}{1+2v_1} (u^T \Gamma^2 u-M^2) > 0,\quad l=\pm F. \nonumber\eea

\item If $v_2 \neq 0$, then 
\bea && w^T \Gamma^2 w-(M+l)w^T \Gamma w + M l \cdot w^T w\nonumber \\
=&& \sum_{k=-F}^F \frac{k^2-(M+l)k+M l}{1+2v_1+2k v_2}\|u_k\|_2^2\nonumber \\
=&& \sum_{k=-F}^F \left[\frac{k}{2v_2}-\frac{M}{2v_2} +
\frac{(1+2v_1 + 2l v_2)(M-k)}{2v_2 (1+2v_1 + 2k v_2)}\right] \|u_k\|_2^2 \nonumber\\
=&& \frac{1}{2v_2}u^T \Gamma u - \frac{M}{2v_2}u^T u 
+ \sum_{k=-F}^F \frac{(1+2v_1 + 2l v_2)(M-k)}{2v_2 (1+2v_1 + 2k v_2)} \|u_k\|_2^2\nonumber \\
=&& \frac{1+2v_1 + 2l v_2}{2} \sum_{k=-F}^F \frac{M-k}{v_2(1+2v_1 + 2k v_2)} \|u_k\|_2^2 ,
 \quad l=\pm F.\nonumber \eea
Noticing that 
\be \small \frac{M-k}{v_2(1+2v_1 + 2k v_2)} - \frac{M-k}{v_2(1+2v_1 +2M v_2)} = \frac{2(M-k)^2}{(1+2v_1 + 2k v_2)(1+2v_1 +2M v_2)} > 0, \ee
we have 
\bea && w^T \Gamma^2 w-(M+l)w^T \Gamma w + M l \cdot w^T w \nonumber\\
>&& \frac{1+2v_1 + 2l v_2}{2} \sum_{k=-F}^F \frac{M-k}{v_2(1+2v_1 + 2M v_2)} \|u_k\|_2^2 \nonumber\\
=&& \frac{1+2v_1 + 2l v_2}{2v_2(1+2v_1 + 2M v_2)}(M u^T u - u^T \Gamma u) 
= 0, \quad l=\pm F. \nonumber\eea
\end{itemize}   

On one hand, above analysis shows $\varphi (\M\times \mathcal{N})\subset S$; 
on the other hand, Lemma 
\ref{lemma:r3} indicates that for any $w\in S$, there exists a unique $u=\psi (w), 
v=(\mu, \lambda)^T$ such that $\varphi (u,v)=w$. Hence $\varphi$ is a bijection from 
$\M \times \mathcal{N}$ to $S$. It is straightforward to see that $\varphi$ and $\varphi^{-1}$ 
are both smooth functions, and $\varphi(u,\mathbf{0})=u, \forall u\in \M$. Thus from 
Theorem \ref{theorem:retr} we know that the closed-form retraction, given by 
\be R_u(\xi_u) := \pi_1 (\varphi^{-1}(u+\xi_u)), \ee
is a retraction on $\M$.
\end{proof}

\section{Numerical Experiments}
\label{sect:result}

In this section, we first compare the performance of Algorithm \ref{alg:ARNT} with 
RGBB and the Riemannian trust region method (RTR) \cite{AbsilRTR} 
by testing some BEC examples. RGBB and RTR are also expedited with the mesh 
refinement technique.
We present numerical results of these algorithms under 
the three different retractions defined in Section \ref{sect:retr}. 
Then we apply Algorithm \ref{alg:ARNT} to compute the ground states of 
spin-2 and spin-3 BEC with different parameters. 
All codes are written in MATLAB. All experiments were performed on a workstation with 
Intel Xenon E5-2680 v3 processors at 2.50GHz($\times 12$) and 128GB memory 
running CentOS 6.8 and MATLAB R2018b.

In the  spin-1 BEC, the initial data is chosen as $\Phi_0(\bx)=U\phi_0(\bx)$, where 
\be
\phi_0(\bx) = \frac1{\pi^{d/4}} e^{-(x_1^2+\cdots+x_d^2)/2},\quad \bx\in\R^d,
\ee
 $U=\left(\frac{\sqrt{1+3M}}{2}, \sqrt{\frac{1-M}{2}}, \frac{\sqrt{1-M}}{2}\right)^T$ for the ferromagnetic interaction ($\beta_1\leq 0$);
$U=\left(\sqrt{\frac{1+M}{2}}, 0, \sqrt{\frac{1-M}{2}}\right)^T$ for the antiferromagnetic interaction ($\beta_1>0$) \cite{Bao1}.

In the spin-2 BEC, the initial data is chosen as $\Phi_0(\bx)=U\phi_0(\bx)$, where 
\be
U=\left(\frac{m_1^4}{16}, \frac{m_1^3m_2}{8}, \frac{\sqrt{6}m_1^2m_2^2}{16}, \frac{m_1m_2^3}{8}, \frac{m_2^4}{16}\right)^T
\ee
with $m_1=\sqrt{2+M}$ and $m_2=\sqrt{2-M}$ for the ferromagnetic interaction ($\beta_1<0$ and $\beta_2>20\beta_1$).
And $U=\left(\frac{\sqrt{2+M}}{2}, 0,0,0, \frac{\sqrt{2-M}}{2}\right)^T$ for the nematic interaction ($\beta_2<0$ and $\beta_2<20\beta_1$), $U=\left(\sqrt{\frac{M+1}{3}}, 0,0, \sqrt{\frac{2-M}{3}}, 0\right)^T$ for the cyclic interaction ($\beta_1>0$ and $\beta_2>0$) \cite{Bao-Tang-Yuan}.
In the spin-3 BEC, the initial data is chosen as $\Phi_0(\bx)=U\phi_0(\bx)$, where $U\in\R^7$ is taken as the random vector.
In all the examples, we take $p=q=0$.

\subsection{Performance of algorithms}
\label{sect:perform}

In RGBB we used all of the default parameters. As for RTR, we added a rule 
$\| r_{j+1}\|_2 \leq \min \{ 0.1, 0.1\| r_0 \|_2\}$ into the stopping criterion of 
the truncated CG method. All other default settings of RTR were used. 
For ARNT, we set $\eta_1 = 0.01, \eta_2 = 0.9, \gamma_0 = 0.2, \gamma_1 = 1, 
\gamma_2 = 10$, and $\sigma_k = \hat{\sigma}_k\|\mathrm{grad}\tilde E(u_k)\|_2$, 
where $\hat{\sigma}_k$ is updated by the procedure in Algorithm \ref{alg:ARNT} with 
$\hat{\sigma}_0 = 1$. 
%The parameters in Algorithm \ref{alg:CG} are chosen as: 
%$\rho = 10^{-4}, \delta = 0.2, \theta = 1$, and $T=0.1$. 
Furthermore, when an estimation 
of the absolute value of the negative curvature, denoted by $\sigma_{est}$, is available at 
the $k$-th subproblem, we can calculate 
\be \sigma_{k+1}^{new} = \max \{ \sigma_{k+1}, \sigma_{est}+\tilde{\gamma} \}, \ee
with some small $\tilde{\gamma}\geq 0$. Then the parameter $\sigma_{k+1}$ is reset to 
$\sigma_{k+1}^{new}$. 

On the finest mesh, all algorithms terminate when either $\| \mathrm{grad}\tilde E(u_k)\|_2 \leq 
10^{-6}$ or the number of iterations reaches 10000, while on the coarse meshes they all 
terminate when $\| \mathrm{grad}E(x_k)\|_2 \leq 10^{-5}$. 
In the implementation of ARNT, RGBB stops when either 
$\| \mathrm{grad}\tilde E(u_k)\|_2 \leq 10^{-2}$ or the number of iterations 
reaches 2000. The maximum number of inner iterations in ARNT is chosen adaptively 
depending on $\| \mathrm{grad}\tilde E(u_k)\|_2$.

In the subsequent tables, the columns ``f", ``nrmG" and ``time" display the final objective 
function value, the final norm of the Riemannian gradient and the total CPU time each 
algorithm spent to reach the stopping criterion. The column ``iter" reports the number 
of iterations (the average numbers of inner iterations) on the finest mesh. 
The choice of retractions is shown in the column 
``retr", where R1, R2 and R3 denote the projective retraction, the orthogonal retraction and 
the closed-form retraction, respectively. 

We present results of following cases for spin-1, spin-2 and spin-3 BEC: 
\begin{itemize}
\item Spin-1 BEC \cite{Bao1}
	\begin{itemize}
%	\item[~] 1D. $V(x) = \frac{1}{2} x^2+ 25\sin^2 (\frac{\pi x}{4})$, 
%	$\beta_0 = 885$, $\beta_1 = -4.1$, $U = [-16, 16], n = 2^{11}.$
	\item[~] 2D. Antiferromagnetic case. $V(x,y) = \frac{1}{2} (x^2 + y^2) + 10[\sin^2 (\frac{\pi x}{4})+\sin^2 
	(\frac{\pi y}{4})]$, $\beta_0 = 300$, $\beta_1 = 100$, $U = [-16, 16]\times[-16, 16], 
	n = 2^9.$
	\item[~] 3D. Ferromagnetic case. $V(x,y,z)=\frac{1}{2} \sum_{\alpha=x,y,z}\left(\alpha^2+200\sin^2 (\frac{\pi \alpha}{2})\right)$, $\beta_0 = 880$, 
	$\beta_1 = -4.1$, $U = [-16, 16]\times[-16, 16]\times[-16, 16], n = 2^8.$
	\end{itemize}
\item Spin-2 BEC \cite{Gautam}
	\begin{itemize}
%	\item[~] 1D.  Ferromagnetic case. $V(x) = \frac{1}{2} x^2+25\sin^2 (\frac{\pi x}{4})$, $\beta_0 = 130.6$, 
%	$\beta_1 = -25.4$, $\beta_2 = -125.3$, $U = [-8, 8], n = 2^8.$
	\item[~] 2D. Antiferromagnetic case. $V(x,y)=\frac{1}{2}(x^2+ y^2)+10\left[\sin^2 (\frac{\pi x}{2})+\sin^2 
	(\frac{\pi y}{2})\right]$, $\beta_0 = 243$, $\beta_1 = 12.1$, $\beta_2 = -13$, 
	$U = [-8, 8]\times [-8, 8], n = 2^8.$
	\item[~] 3D. Cyclic case. $V(x,y,z)=\frac{1}{2}(x^2+y^2+z^2)+100\left[\sin^2 (\frac{\pi x}{2})+
	\sin^2 (\frac{\pi y}{2})+\sin^2(\frac{\pi z}{2})\right]$, $\beta_0 = 183.9$, 
	$\beta_1 = 26.8$, $\beta_2 = 134.7$, $U = [-16, 16]\times [-16, 16]\times [-16, 16], 
	n = 2^8.$
	\end{itemize}
\item Spin-3 BEC
	\begin{itemize}
	\item[~] 2D. $V(x,y)=\frac{1}{2}(x^2+ y^2)+10\left[\sin^2 (\frac{\pi x}{2})+
	\sin^2 (\frac{\pi y}{2})\right]$, $\beta_0 = 100$, $\beta_1 = 1, \beta_2 = 10, 
	\beta_3 = 1$, $U = [-8, 8]\times [-8, 8], n = 2^8.$
	\item[~] 3D. $V(x,y,z)=\frac{1}{2}(x^2+y^2+z^2)+100\left[\sin^2 (\frac{\pi x}
	{2})+\sin^2 (\frac{\pi y}{2})+\sin^2(\frac{\pi z}{2})\right]$, $\beta_0 = 100$, 
	$\beta_1 = 1, \beta_2 = 10, \beta_3 = 1$, $U = [-8, 8]\times [-8, 8]\times 
	[-8, 8], n = 2^7.$
	\end{itemize}	
\end{itemize}

\setlength{\tabcolsep}{1.6pt}

\begin{table}[htb] 
\centering 
\caption{Numerical results of spin-1 BEC in 2D} 
\begin{tabular}{|c|cccc|cccc|cccc|} 
\hline 
~ & \multicolumn{4}{|c|}{ARNT} & \multicolumn{4}{|c|}{RGBB} & \multicolumn{4}{|c|}{RTR} \\ \hline 
retr& f & nrmG & iter & time& f & nrmG & iter & time& f & nrmG & iter & time \\ \hline 
\multicolumn{13}{|c|}{$M = 0.0$} \\ \hline 
R1 & 15.1032 & 7.8e-07 & 4 (38) & 17.0 & 15.1032 & 8.3e-07 & 239 & 18.1 & 15.1032 & 6.2e-07 & 20 (17) & 31.6 \\ \hline 
R2 & 15.1032 & 7.7e-07 & 4 (38) & 15.3 & 15.1032 & 6.3e-07 & 257 & 15.1 & 15.1032 & 6.3e-07 & 20 (17) & 31.0 \\ \hline 
R3 & 15.1032 & 4.4e-07 & 4 (35) & 15.9 & 15.1032 & 3.3e-07 & 255 & 15.9 & 15.1032 & 6.2e-07 & 20 (17) & 28.9 \\ \hline 
\multicolumn{13}{|c|}{$M = 0.2$} \\ \hline 
R1 & 15.1411 & 6.6e-07 & 4 (44) & 19.9 & 15.1411 & 9.3e-07 & 258 & 21.5 & 15.1411 & 2.6e-07 & 21 (26) & 54.9 \\ \hline 
R2 & 15.1411 & 8.1e-07 & 4 (42) & 18.6 & 15.1411 & 1.0e-06 & 254 & 15.1 & 15.1411 & 2.7e-07 & 21 (26) & 51.2 \\ \hline 
R3 & 15.1411 & 7.6e-07 & 4 (42) & 18.6 & 15.1411 & 7.9e-07 & 261 & 18.2 & 15.1411 & 2.6e-07 & 21 (26) & 50.8 \\ \hline 
\multicolumn{13}{|c|}{$M = 0.5$} \\ \hline 
R1 & 15.3436 & 6.7e-07 & 4 (64) & 27.2 & 15.3436 & 8.1e-07 & 431 & 36.4 & 15.3436 & 2.1e-07 & 21 (31) & 61.2 \\ \hline 
R2 & 15.3436 & 3.5e-07 & 4 (67) & 24.9 & 15.3436 & 9.1e-07 & 421 & 26.4 & 15.3436 & 1.9e-07 & 21 (31) & 61.4 \\ \hline 
R3 & 15.3436 & 5.3e-07 & 4 (64) & 25.8 & 15.3436 & 9.0e-07 & 429 & 26.9 & 15.3436 & 1.9e-07 & 21 (31) & 63.5 \\ \hline 
\multicolumn{13}{|c|}{$M = 0.9$} \\ \hline 
R1 & 15.9621 & 7.7e-07 & 4 (51) & 23.8 & 15.9621 & 8.8e-07 & 323 & 28.4 & 15.9621 & 3.9e-07 & 21 (26) & 54.3 \\ \hline 
R2 & 15.9621 & 7.5e-07 & 4 (52) & 21.7 & 15.9621 & 1.0e-06 & 279 & 18.1 & 15.9621 & 3.9e-07 & 21 (26) & 51.3 \\ \hline 
R3 & 15.9621 & 5.3e-07 & 4 (55) & 23.5 & 15.9621 & 7.6e-07 & 483 & 31.8 & 15.9621 & 3.6e-07 & 21 (26) & 51.6 \\ \hline 
\end{tabular} \label{tb:1-1-2}
\end{table} 

\begin{table}[htb] 
\centering 
\caption{Numerical results of spin-1 BEC in 3D} 
\begin{tabular}{|c|cccc|cccc|cccc|} 
\hline 
~ & \multicolumn{4}{|c|}{ARNT} & \multicolumn{4}{|c|}{RGBB} & \multicolumn{4}{|c|}{RTR} \\ \hline 
retr& f & nrmG & iter & time& f & nrmG & iter & time& f & nrmG & iter & time \\ \hline 
\multicolumn{13}{|c|}{$M = 0.0$} \\ \hline 
R1 & 55.4362 & 6.9e-07 & 4 (34) & 1187.6 & 55.4362 & 1.0e-06 & 720 & 7197.9 & 55.4362 & 7.2e-07 & 17 (15) & 1484.0 \\ \hline 
R2 & 55.4362 & 7.5e-07 & 4 (34) & 1042.3 & 55.4362 & 9.5e-07 & 384 & 2872.6 & 55.4362 & 7.2e-07 & 17 (15) & 1306.8 \\ \hline 
R3 & 55.4362 & 5.6e-07 & 4 (35) & 1210.5 & 55.4362 & 7.0e-07 & 227 & 1169.3 & 55.4362 & 7.0e-07 & 17 (15) & 1540.0 \\ \hline 
\multicolumn{13}{|c|}{$M = 0.2$} \\ \hline 
R1 & 55.4362 & 4.2e-07 & 4 (40) & 1568.7 & 55.4363 & 9.0e-05 & 10000 & 47615.4 & 55.4362 & 7.3e-07 & 17 (15) & 1523.8 \\ \hline 
R2 & 55.4362 & 5.8e-07 & 4 (33) & 2041.6 & 55.4363 & 9.8e-05 & 10000 & 42311.5 & 55.4362 & 7.1e-07 & 17 (15) & 1309.3 \\ \hline 
R3 & 55.4362 & 3.2e-07 & 4 (41) & 2626.7 & 55.4363 & 9.1e-05 & 10000 & 48975.6 & 55.4362 & 7.1e-07 & 17 (15) & 1538.8 \\ \hline 
\multicolumn{13}{|c|}{$M = 0.5$} \\ \hline 
R1 & 55.4362 & 5.6e-07 & 4 (41) & 1395.6 & 55.4362 & 8.3e-07 & 201 & 1242.1 & 55.4362 & 7.0e-07 & 17 (15) & 1514.3 \\ \hline 
R2 & 55.4362 & 3.2e-07 & 4 (41) & 1392.8 & 55.4362 & 3.1e-07 & 310 & 1846.4 & 55.4362 & 7.0e-07 & 17 (15) & 1309.3 \\ \hline 
R3 & 55.4362 & 5.7e-07 & 4 (36) & 1394.5 & 55.4362 & 1.0e-06 & 295 & 1603.6 & 55.4362 & 7.2e-07 & 17 (15) & 1532.3 \\ \hline 
\multicolumn{13}{|c|}{$M = 0.9$} \\ \hline 
R1 & 55.4362 & 5.4e-07 & 4 (39) & 1421.1 & 55.4362 & 9.8e-07 & 444 & 4093.6 & 55.4362 & 6.8e-07 & 17 (15) & 1527.9 \\ \hline 
R2 & 55.4362 & 2.9e-07 & 4 (41) & 1323.6 & 55.4362 & 8.7e-07 & 410 & 2916.9 & 55.4362 & 7.2e-07 & 17 (15) & 1297.3 \\ \hline 
R3 & 55.4362 & 6.8e-07 & 4 (34) & 1350.7 & 55.4362 & 9.7e-07 & 236 & 1460.1 & 55.4362 & 6.9e-07 & 17 (15) & 1548.9 \\ \hline 
\end{tabular} \label{tb:1-1-3}
\end{table} 

%\begin{table}[htb] 
%\centering 
%\caption{Numerical results of spin-2 BEC in 1D} 
%\begin{tabular}{|c|cccc|cccc|cccc|} 
%\hline 
%~ & \multicolumn{4}{|c|}{ARNT} & \multicolumn{4}{|c|}{RGBB} & \multicolumn{4}{|c|}{RTR} \\ \hline 
%retr& f & nrmG & iter & time& f & nrmG & iter & time& f & nrmG & iter & time \\ \hline 
%\multicolumn{13}{|c|}{$M = 0.0$} \\ \hline 
%R1 & 10.3700 & 5.0e-07 & 4 (17) & 0.5 & 10.3700 & 9.7e-07 & 110 & 0.9 & 10.3703 & 3.6e-07 & 13 (15) & 0.2 \\ \hline 
%R2 & 10.3700 & 5.0e-07 & 4 (17) & 0.5 & 10.3700 & 5.6e-07 & 101 & 0.9 & 10.3703 & 3.6e-07 & 13 (15) & 0.2 \\ \hline 
%R3 & 10.3700 & 4.9e-07 & 4 (17) & 0.3 & 10.3700 & 9.3e-07 & 95 & 0.3 & 10.3703 & 3.6e-07 & 13 (15) & 0.2 \\ \hline 
%\multicolumn{13}{|c|}{$M = 0.5$} \\ \hline 
%R1 & 10.3700 & 4.9e-07 & 4 (17) & 0.3 & 10.3700 & 5.0e-07 & 97 & 0.5 & 10.3703 & 3.5e-07 & 13 (15) & 0.2 \\ \hline 
%R2 & 10.3700 & 4.9e-07 & 4 (17) & 0.4 & 10.3700 & 9.2e-07 & 105 & 0.5 & 10.3703 & 3.5e-07 & 13 (15) & 0.2 \\ \hline 
%R3 & 10.3700 & 4.9e-07 & 4 (17) & 0.2 & 10.3700 & 9.4e-07 & 94 & 0.2 & 10.3703 & 3.5e-07 & 13 (15) & 0.2 \\ \hline 
%\multicolumn{13}{|c|}{$M = 1.5$} \\ \hline 
%R1 & 10.3700 & 4.9e-07 & 4 (17) & 0.3 & 10.3700 & 9.7e-07 & 100 & 0.5 & 10.3703 & 3.5e-07 & 13 (15) & 0.2 \\ \hline 
%R2 & 10.3700 & 4.9e-07 & 4 (17) & 0.3 & 10.3700 & 9.9e-07 & 107 & 0.5 & 10.3703 & 3.5e-07 & 13 (15) & 0.2 \\ \hline 
%R3 & 10.3700 & 4.9e-07 & 4 (17) & 0.2 & 10.3700 & 9.4e-07 & 94 & 0.2 & 10.3703 & 3.5e-07 & 13 (15) & 0.2 \\ \hline 
%\end{tabular} \label{tb:1-2-1}
%\end{table} 

\begin{table}[htb] 
\centering 
\caption{Numerical results of spin-2 BEC in 2D} 
\begin{tabular}{|c|cccc|cccc|cccc|} 
\hline 
~ & \multicolumn{4}{|c|}{ARNT} & \multicolumn{4}{|c|}{RGBB} & \multicolumn{4}{|c|}{RTR} \\ \hline 
retr& f & nrmG & iter & time& f & nrmG & iter & time& f & nrmG & iter & time \\ \hline 
\multicolumn{13}{|c|}{$M = 0.0$} \\ \hline 
R1 & 14.3386 & 5.9e-07 & 4 (32) & 8.3 & 14.3386 & 9.0e-07 & 280 & 12.1 & 14.3386 & 9.3e-07 & 17 (18) & 18.3 \\ \hline 
R2 & 14.3386 & 5.9e-07 & 4 (32) & 8.4 & 14.3386 & 8.0e-07 & 238 & 9.8 & 14.3386 & 9.3e-07 & 17 (18) & 17.2 \\ \hline 
R3 & 14.3386 & 5.9e-07 & 4 (32) & 8.2 & 14.3386 & 9.0e-07 & 225 & 8.3 & 14.3386 & 9.3e-07 & 17 (18) & 17.2 \\ \hline 
\multicolumn{13}{|c|}{$M = 0.5$} \\ \hline 
R1 & 14.3730 & 3.8e-07 & 4 (59) & 13.5 & 14.3730 & 1.3e-07 & 346 & 14.9 & 14.3730 & 4.7e-07 & 18 (22) & 25.4 \\ \hline 
R2 & 14.3730 & 3.8e-07 & 4 (59) & 13.8 & 14.3730 & 9.9e-07 & 323 & 13.9 & 14.3730 & 4.7e-07 & 18 (22) & 24.3 \\ \hline 
R3 & 14.3730 & 3.8e-07 & 4 (59) & 13.5 & 14.3730 & 9.6e-08 & 523 & 32.1 & 14.3730 & 4.7e-07 & 18 (22) & 23.1 \\ \hline 
\multicolumn{13}{|c|}{$M = 1.5$} \\ \hline 
R1 & 14.6754 & 4.3e-07 & 4 (62) & 13.5 & 14.6754 & 9.5e-07 & 462 & 19.0 & 14.6754 & 2.8e-07 & 18 (28) & 30.5 \\ \hline 
R2 & 14.6754 & 4.3e-07 & 4 (62) & 14.2 & 14.6754 & 9.9e-07 & 519 & 31.8 & 14.6754 & 2.8e-07 & 18 (28) & 28.5 \\ \hline 
R3 & 14.6754 & 4.3e-07 & 4 (62) & 13.8 & 14.6754 & 8.5e-07 & 402 & 19.4 & 14.6754 & 2.8e-07 & 18 (28) & 28.5 \\ \hline 
\end{tabular} \label{tb:1-2-2}
\end{table} 

\begin{table}[htb]
\centering 
\caption{Numerical results of spin-2 BEC in 3D} 
\begin{tabular}{|c|cccc|cccc|cccc|} 
\hline 
~ & \multicolumn{4}{|c|}{ARNT} & \multicolumn{4}{|c|}{RGBB} & \multicolumn{4}{|c|}{RTR} \\ \hline 
retr& f & nrmG & iter & time& f & nrmG & iter & time& f & nrmG & iter & time \\ \hline 
\multicolumn{13}{|c|}{$M = 0.0$} \\ \hline 
R1 & 46.2917 & 5.9e-07 & 4 (36) & 2179.1 & 46.2917 & 4.7e-07 & 358 & 4192.9 & 46.2917 & 3.6e-07 & 18 (15) & 2905.1 \\ \hline 
R2 & 46.2917 & 6.6e-07 & 4 (40) & 2163.8 & 46.2917 & 9.7e-07 & 346 & 3667.8 & 46.2917 & 3.4e-07 & 18 (15) & 2720.1 \\ \hline 
R3 & 46.2917 & 6.5e-07 & 4 (45) & 2534.3 & 46.2917 & 7.7e-07 & 356 & 4219.3 & 46.2917 & 2.7e-07 & 18 (15) & 2937.6 \\ \hline 
\multicolumn{13}{|c|}{$M = 0.5$} \\ \hline 
R1 & 45.7403 & 6.2e-07 & 5 (81) & 4824.5 & 45.7403 & 9.3e-07 & 1129 & 9464.9 & 45.7403 & 5.5e-07 & 19 (29) & 5432.8 \\ \hline 
R2 & 45.7403 & 8.4e-07 & 5 (70) & 4176.5 & 45.7403 & 9.2e-07 & 1050 & 7630.2 & 45.7403 & 4.7e-07 & 19 (28) & 5277.9 \\ \hline 
R3 & 45.7403 & 7.7e-07 & 5 (76) & 4511.5 & 45.7403 & 9.6e-07 & 1296 & 12903.3 & 45.7403 & 5.3e-07 & 19 (28) & 5483.4 \\ \hline 
\multicolumn{13}{|c|}{$M = 1.5$} \\ \hline 
R1 & 46.8619 & 7.2e-07 & 4 (56) & 3032.4 & 46.8619 & 9.8e-07 & 391 & 3182.1 & 46.8619 & 3.3e-07 & 19 (22) & 4371.3 \\ \hline 
R2 & 46.8619 & 6.6e-07 & 4 (56) & 3038.1 & 46.8619 & 8.9e-07 & 380 & 2996.6 & 46.8619 & 3.4e-07 & 19 (22) & 4142.4 \\ \hline 
R3 & 46.8619 & 5.8e-07 & 4 (55) & 3151.1 & 46.8619 & 9.3e-07 & 515 & 6125.7 & 46.8619 & 3.5e-07 & 19 (22) & 4370.7 \\ \hline 
\end{tabular} \label{tb:1-2-3}
\end{table} 

\begin{table}[htb] 
\centering 
\caption{Numerical results of spin-3 BEC in 2D} 
\begin{tabular}{|c|cccc|cccc|cccc|} 
\hline 
~ & \multicolumn{4}{|c|}{ARNT} & \multicolumn{4}{|c|}{RGBB} & \multicolumn{4}{|c|}{RTR} \\ \hline 
retr& f & nrmG & iter & time& f & nrmG & iter & time& f & nrmG & iter & time \\ \hline 
\multicolumn{13}{|c|}{$M = 0.0$} \\ \hline 
R1 & 11.8279 & 6.8e-07 & 4 (58) & 29.8 & 11.8279 & 6.8e-07 & 634 & 50.3 & 11.8279 & 5.2e-07 & 28 (159) & 359.0 \\ \hline 
R2 & 11.8279 & 7.1e-07 & 4 (49) & 24.4 & 11.8279 & 8.6e-07 & 469 & 39.5 & 11.8279 & 1.4e-07 & 18 (58) & 83.2 \\ \hline 
R3 & 11.8279 & 6.8e-07 & 4 (58) & 25.9 & 11.8279 & 8.6e-07 & 496 & 32.9 & 11.8279 & 5.9e-07 & 32 (184) & 476.4 \\ \hline 
\multicolumn{13}{|c|}{$M = 0.5$} \\ \hline 
R1 & 11.8334 & 4.9e-07 & 4 (75) & 32.7 & 11.8334 & 7.8e-07 & 577 & 48.0 & 11.8334 & 1.8e-07 & 18 (79) & 124.2 \\ \hline 
R2 & 11.8334 & 5.2e-07 & 4 (75) & 35.4 & 11.8334 & 8.7e-07 & 472 & 37.7 & 11.8334 & 1.7e-07 & 18 (79) & 127.5 \\ \hline 
R3 & 11.8334 & 5.9e-07 & 4 (73) & 31.0 & 11.8334 & 9.5e-07 & 572 & 40.0 & 11.8334 & 1.9e-07 & 18 (42) & 63.5 \\ \hline 
\multicolumn{13}{|c|}{$M = 1.5$} \\ \hline 
R1 & 11.8780 & 4.0e-07 & 5 (134) & 51.6 & 11.8780 & 1.0e-06 & 2874 & 190.9 & 11.8780 & 3.0e-07 & 18 (51) & 78.6 \\ \hline 
R2 & 11.8780 & 3.6e-07 & 6 (150) & 79.6 & 11.8780 & 9.7e-07 & 857 & 74.9 & 11.8780 & 2.7e-07 & 18 (61) & 86.9 \\ \hline 
R3 & 11.8780 & 6.2e-07 & 4 (95) & 36.5 & 11.8780 & 9.9e-07 & 3123 & 221.6 & 11.8780 & 3.0e-07 & 18 (49) & 69.3 \\ \hline 
\end{tabular} \label{tb:1-3-2}
\end{table} 

\begin{table}[htb] 
\centering 
\caption{Numerical results of spin-3 BEC in 3D} 
\begin{tabular}{|c|cccc|cccc|cccc|} 
\hline 
~ & \multicolumn{4}{|c|}{ARNT} & \multicolumn{4}{|c|}{RGBB} & \multicolumn{4}{|c|}{RTR} \\ \hline 
retr& f & nrmG & iter & time& f & nrmG & iter & time& f & nrmG & iter & time \\ \hline 
\multicolumn{13}{|c|}{$M = 0.0$} \\ \hline 
R1 & 42.9752 & 5.1e-07 & 5 (164) & 4243.2 & 42.9752 & 9.6e-07 & 496 & 1328.1 & 42.9752 & 9.5e-08 & 18 (23) & 794.0 \\ \hline 
R2 & 42.9752 & 5.9e-07 & 4 (92) & 4405.6 & 42.9774 & 1.6e-04 & 10000 & 15255.9 & 42.9752 & 9.9e-08 & 18 (30) & 1027.9 \\ \hline 
R3 & 42.9752 & 9.9e-07 & 40 (154) & 16610.8 & 42.9767 & 2.2e-03 & 10000 & 14740.1 & 42.9752 & 4.7e-07 & 48 (98) & 8453.7 \\ \hline 
\multicolumn{13}{|c|}{$M = 0.5$} \\ \hline 
R1 & 42.9822 & 8.8e-07 & 4 (115) & 1050.1 & 42.9822 & 9.4e-07 & 864 & 1370.7 & 42.9822 & 7.7e-07 & 96 (159) & 27219.4 \\ \hline 
R2 & 42.9822 & 8.8e-07 & 4 (115) & 1050.4 & 42.9822 & 1.0e-06 & 1150 & 1956.0 & 42.9822 & 6.9e-07 & 21 (40) & 1601.1 \\ \hline 
R3 & 42.9822 & 8.9e-07 & 4 (115) & 1013.1 & 42.9822 & 1.0e-06 & 1328 & 2128.7 & 42.9822 & 6.7e-07 & 21 (40) & 1539.8 \\ \hline 
\multicolumn{13}{|c|}{$M = 1.5$} \\ \hline 
R1 & 43.0399 & 2.8e-07 & 5 (151) & 1959.2 & 43.0400 & 4.2e-03 & 10000 & 15199.3 & 43.0399 & 1.2e-07 & 82 (379) & 55201.7 \\ \hline 
R2 & 43.0399 & 2.3e-07 & 5 (156) & 2100.3 & 43.0399 & 2.1e-03 & 10000 & 15261.4 & 43.0417 & 1.6e-07 & 22 (36) & 1479.3 \\ \hline 
R3 & 43.0399 & 2.5e-07 & 5 (151) & 1683.8 & 43.0399 & 9.1e-07 & 1162 & 1802.2 & 43.0417 & 1.6e-07 & 22 (37) & 1485.9 \\ \hline 
\end{tabular} \label{tb:1-3-3}
\end{table} 

The detailed numerical results are reported in Tables \ref{tb:1-1-2}-\ref{tb:1-3-3}. 
%In the 1-dimensional case (Table \ref{tb:1-2-1}), RTR converges to a higher 
%function value than ARNT and RGBB. 
In most cases, all three algorithms 
converge to points with the same function 
values. 
For spin-1 and spin-2 cases,
the choice of different retractions has small impact on the numerical performance, and 
the second-order algorithms ARNT and RTR exhibit higher stability than the first-order algorithm RGBB in response to the change of retractions. 
In the 3D case of spin-3 BEC (Table \ref{tb:1-3-3}), RTR converges to a 
larger function value than ARNT when $M=1.5$  using retractions 
R2 and R3.
Overall, taking both numerical 
stability and time cost into consideration, ARNT shows the best performance.

\subsection{Application in spin-2 BEC}
\label{sect:spin2}

In this section, we apply the ARNT method with the projective retraction to compute the 
ground state of a spin-2 BEC in 1-3 dimensions and under different interactions. Specifically, 
the following cases are considered \cite{Gautam}: 
\begin{itemize}
\item 1D, $V(x) = \frac{1}{2} x^2+25\sin^2 (\frac{\pi x}{4}).$
	\begin{itemize}
	\item[~] Case \rn 1 (ferromagnetic). ~~$\beta_0 = 130.6$, $\beta_1 = -25.4, \beta_2 = -125.3$, 
	$U = [-8, 8], n = 2^8.$
	\item[~] Case \rn 2 (antiferromagnetic). ~$\beta_0 = 243$, $\beta_1 = 12.1, \beta_2 = -13$, 
	$U = [-16, 16], n = 2^9.$
	\item[~] Case \rn 3 (cyclic). $\beta_0 = 183.9$, $\beta_1 = 26.8, \beta_2 = 134.7$, 
	$U = [-16, 16], n = 2^9.$
	\end{itemize}
\item 2D, $V(x,y)=\frac{1}{2}(x^2+y^2)+10\left[\sin^2 (\frac{\pi x}{2})+
\sin^2 (\frac{\pi y}{2})\right]$, $U = [-8, 8]\times [-8, 8], n = 2^8.$
	\begin{itemize}
	\item[~] Case \rn 1 (ferromagnetic). ~~$\beta_0 = 130.6$, $\beta_1 = -25.4, \beta_2 = -125.3$.
	\item[~] Case \rn 2 (antiferromagnetic). ~$\beta_0 = 243$, $\beta_1 = 12.1, \beta_2 = -13$.
	\item[~] Case \rn 3 (cyclic). $\beta_0 = 183.9$, $\beta_1 = 26.8, \beta_2 = 134.7$.
	\end{itemize}
\item 3D, $V(x,y,z)=\frac{1}{2}(x^2+y^2+z^2)+100\left[\sin^2 (\frac{\pi x}{2})+
\sin^2 (\frac{\pi y}{2})+\sin^2(\frac{\pi z}{2})\right]$
	\begin{itemize}
	\item[~] Case \rn 1 (ferromagnetic). ~~$\beta_0 = 130.6$, $\beta_1 = -25.4, \beta_2 = -125.3$, 
	$U = [-8, 8]\times [-16, 16]\times [-16, 16], n = 2^7.$
	\item[~] Case \rn 2 (antiferromagnetic). ~$\beta_0 = 243$, $\beta_1 = 12.1, \beta_2 = -13$, 
	$U = [-16, 16]\times [-16, 16]\times [-16, 16], n = 2^8.$
	\item[~] Case \rn 3 (cyclic). $\beta_0 = 183.9$, $\beta_1 = 26.8, \beta_2 = 134.7$, 
	$U = [-16, 16]\times [-16, 16]\times [-16, 16], n = 2^8.$
	\end{itemize}
\end{itemize}

\begin{table}[htb] 
\centering 
\caption{Ground state energies of spin-2 BEC} 
\begin{tabular}{c|ccc|ccc|ccc} 
\hline 
\multirow{2}*{$M$} & \multicolumn{3}{c|}{1D} & \multicolumn{3}{c|}{2D} & \multicolumn{3}{c}{3D} \\ \cline{2-10}~ & Case \rn 1 & Case \rn 2 & Case \rn 3 & Case \rn 1 & Case \rn 2 & Case \rn 3 & Case \rn 1 & Case \rn 2 & Case \rn 3 \\ \hline 
0.0 & 10.3700 & 25.6185 & 24.1144 & 9.5754 & 14.3386 & 13.9158 & 39.0045 & 46.9770 & 46.2917 \\ 
0.5 & 10.3700 & 25.7372 & 22.9404 & 9.5754 & 14.3730 & 13.5746 & 39.0045 & 47.0301 & 45.7403 \\ 
1.5 & 10.3700 & 26.8415 & 25.4640 & 9.5754 & 14.6754 & 14.2734 & 39.0045 & 47.5117 & 46.8619 \\ 
\hline 
\end{tabular} \label{tb:3-1}
\end{table} 

%\begin{table}[htb] 
%\centering 
%\caption{Ground state energies of spin-2 BEC in 1D} 
%\begin{tabular}{c|ccc} 
%\hline 
%$M$ & Case \rn 1 & Case \rn 2 & Case \rn 3 \\ \hline 
%0.0 & 10.3700300090 & 25.6185042634 & 24.1143586550 \\ 
%0.5 & 10.3700300090 & 25.7372435384 & 22.9403908204 \\ 
%1.5 & 10.3700300090 & 26.8415145164 & 25.4639547287 \\ 
%\hline 
%\end{tabular} \label{tb:3-1}
%\end{table} 
%
%\begin{table}[htb] 
%\centering 
%\caption{Ground state energies of spin-2 BEC in 2D} 
%\begin{tabular}{c|ccc} 
%\hline 
%$M$ & Case \rn 1 & Case \rn 2 & Case \rn 3 \\ \hline 
%0.0 & 9.5753719272 & 14.3386450050 & 13.9157863663 \\ 
%0.5 & 9.5753719272 & 14.3730365308 & 13.5746308356 \\ 
%1.5 & 9.5753719272 & 14.6754064541 & 14.2733711608 \\ 
%\hline 
%\end{tabular} \label{tb:3-2}
%\end{table} 
%
%\begin{table}[htb] 
%\centering 
%\caption{Ground state energies of spin-2 BEC in 3D} 
%\begin{tabular}{c|ccc} 
%\hline 
%$M$ & Case \rn 1 & Case \rn 2 & Case \rn 3 \\ \hline 
%0.0 & 39.0044828481 & 46.9770350126 & 46.2917160579 \\ 
%0.5 & 39.0044828481 & 47.0300525636 & 45.7402842525 \\ 
%1.5 & 39.0044828481 & 47.5117411564 & 46.8619102350 \\ 
%\hline 
%\end{tabular} \label{tb:3-3}
%\end{table} 

The ground state energies in above cases are listed in Table \ref{tb:3-1}. 
Under ferromagnetic interaction (Case \rn 1) the energy remains 
constant when $M$ changes; under antiferromagnetic interaction (Case \rn 2), the energy gets higher 
as $M$ increases; under cyclic interaction (Case \rn 3), the energy decreases first and 
then goes up as $M$ increases.

\begin{figure} 
\centering 
\subfigure{\includegraphics[width=0.32\linewidth]{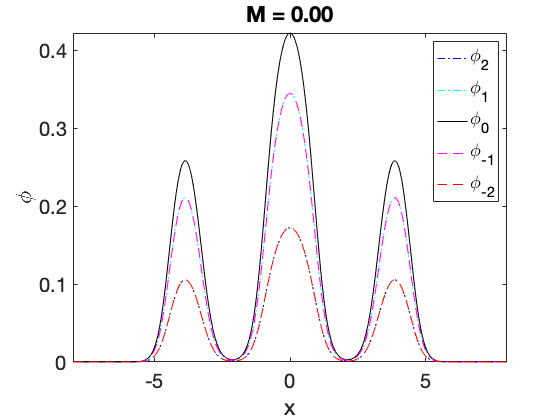}}
\subfigure{\includegraphics[width=0.32\linewidth]{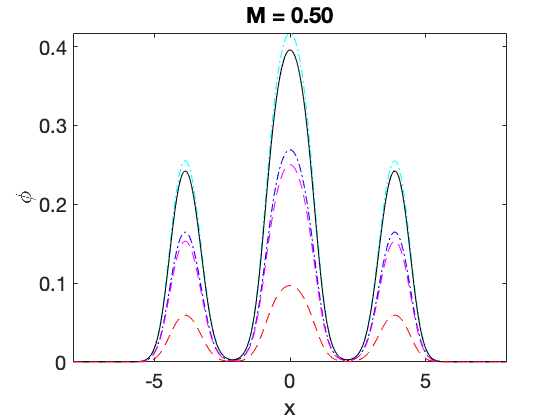}}
\subfigure{\includegraphics[width=0.32\linewidth]{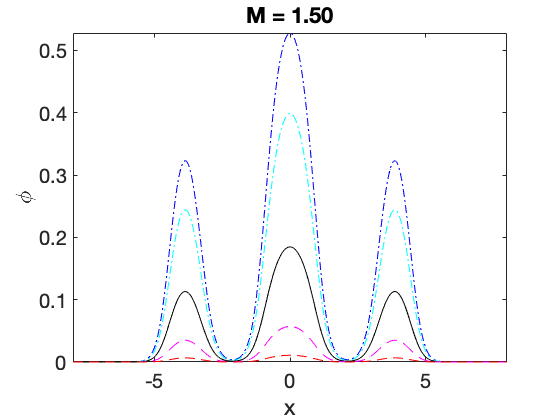}}
\caption{Wave functions of the ground state, i.e. $\phi_2(x)$ (blue dash-dot line), $\phi_1(x)$
(light blue dash-dot line), $\phi_0(x)$(black solid line), $\phi_{-1}(x)$ (purple dashed line) 
 and $\phi_{-2}(x)$ (red dashed line) of a spin-2 BEC for Case \rn 1 in  
1D under different magnetizations $M = 0, 0.5, 1.5$.} 
\label{fig:3-1}
\end{figure}

\begin{figure} 
\centering 
\subfigure{\includegraphics[width=0.32\linewidth]{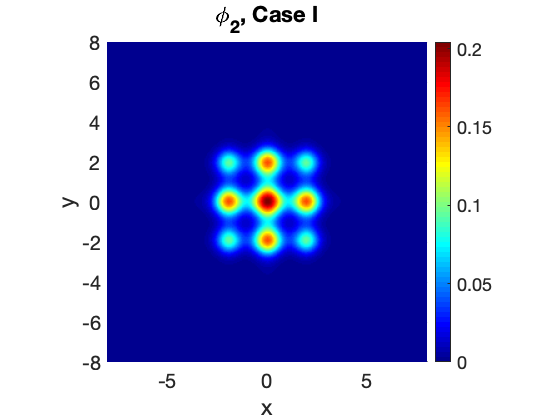}}
\subfigure{\includegraphics[width=0.32\linewidth]{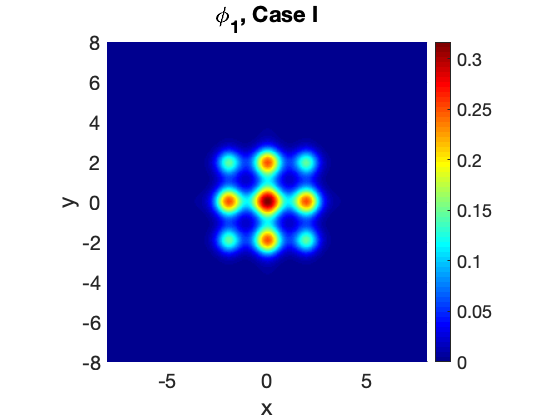}}
\subfigure{\includegraphics[width=0.32\linewidth]{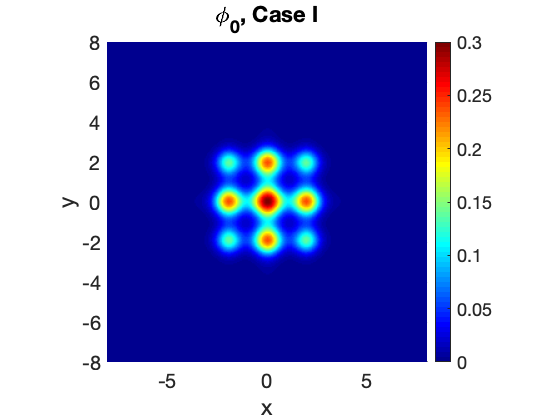}}
\vfill
\subfigure{\includegraphics[width=0.32\linewidth]{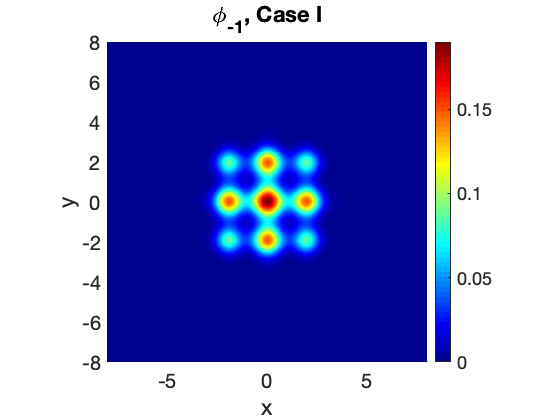}}
\subfigure{\includegraphics[width=0.32\linewidth]{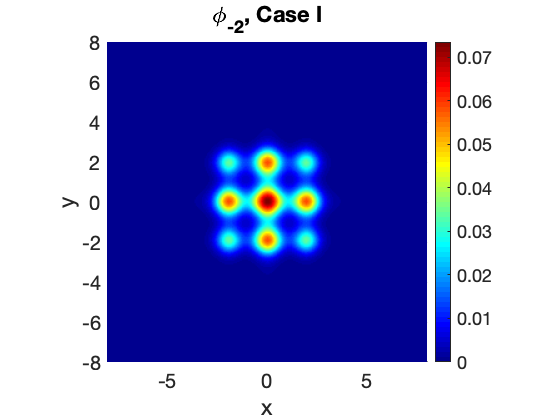}}
\subfigure{\includegraphics[width=0.32\linewidth]{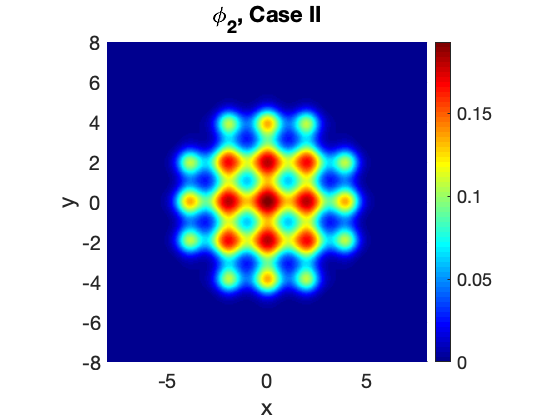}}
\vfill
\subfigure{\includegraphics[width=0.32\linewidth]{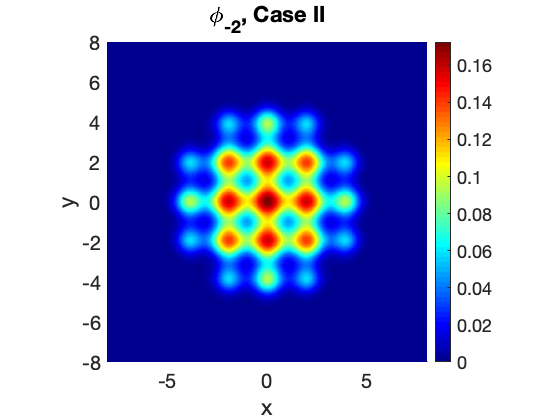}}
\subfigure{\includegraphics[width=0.32\linewidth]{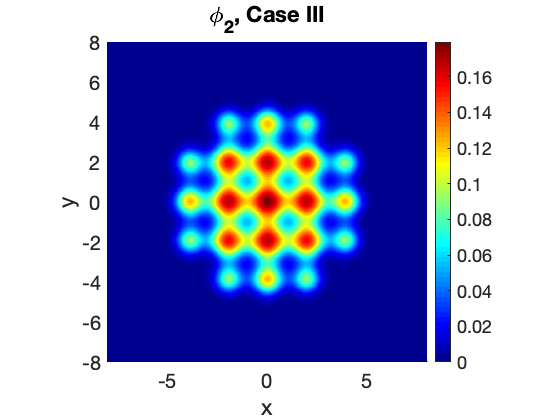}}
\subfigure{\includegraphics[width=0.32\linewidth]{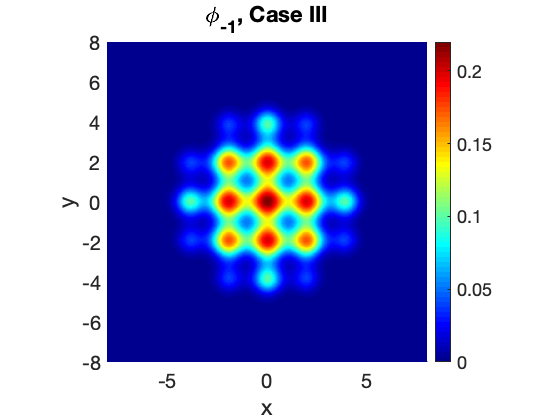}}
\caption{Contour plots for the wave functions of the ground state, i.e. $\phi_2(x,y)$, 
$\phi_1(x,y)$, $\phi_0(x,y)$, $\phi_{-1}(x,y)$, $\phi_{-2}(x,y)$ of a 
spin-2 BEC in 2D with $M = 0.5$ under different interactions. 
In Case \rn 2, the components $\phi_1(x,y)$, $\phi_0(x,y)$, $\phi_{-1}(x,y)$ are zero.
In Case \rn 3, the components $\phi_1(x,y)$, $\phi_0(x,y)$, $\phi_{-2}(x,y)$ are zero.}
\label{fig:3-2}
\end{figure}

\begin{figure} 
\centering 
\subfigure{\includegraphics[width=0.32\linewidth]{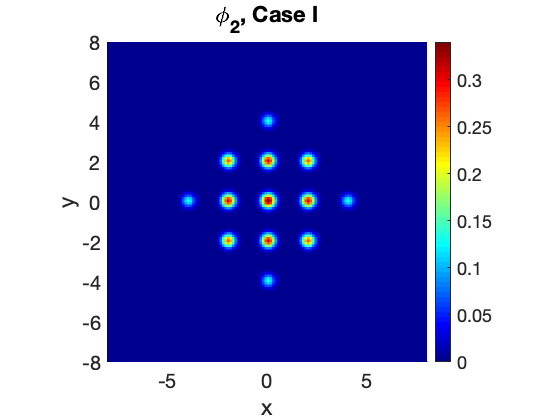}}
\subfigure{\includegraphics[width=0.32\linewidth]{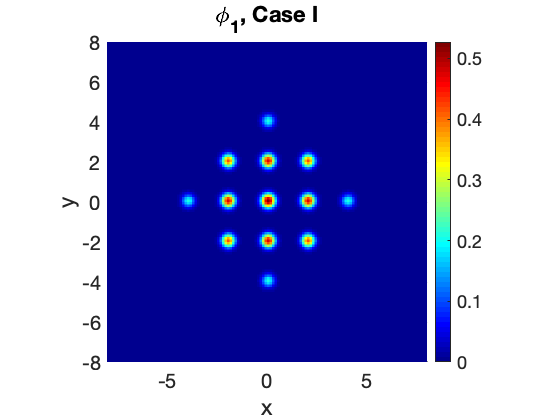}}
\subfigure{\includegraphics[width=0.32\linewidth]{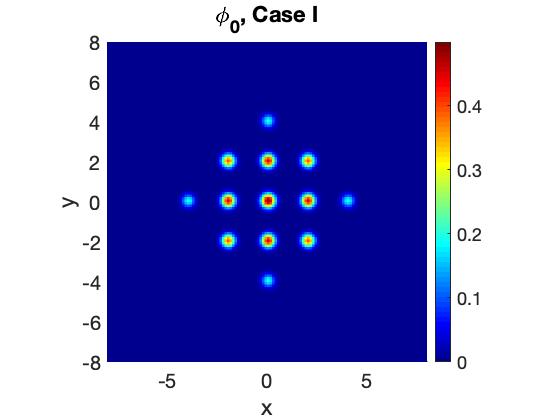}}
\vfill
\subfigure{\includegraphics[width=0.32\linewidth]{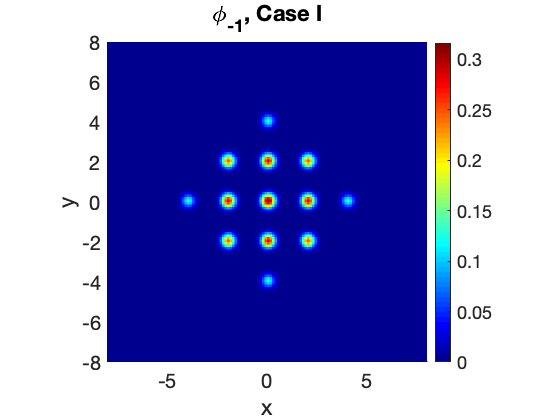}}
\subfigure{\includegraphics[width=0.32\linewidth]{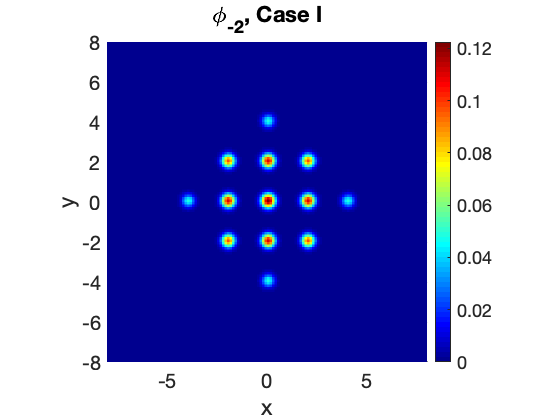}}
\subfigure{\includegraphics[width=0.32\linewidth]{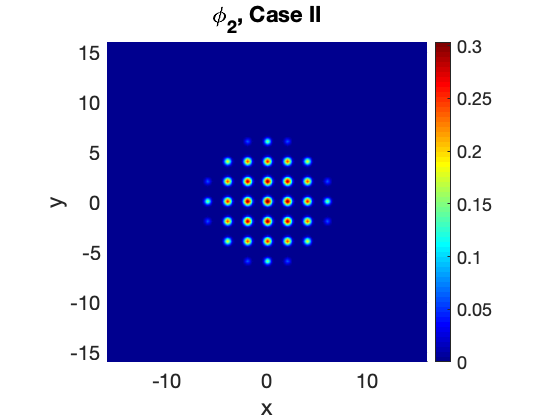}}
\vfill
\subfigure{\includegraphics[width=0.32\linewidth]{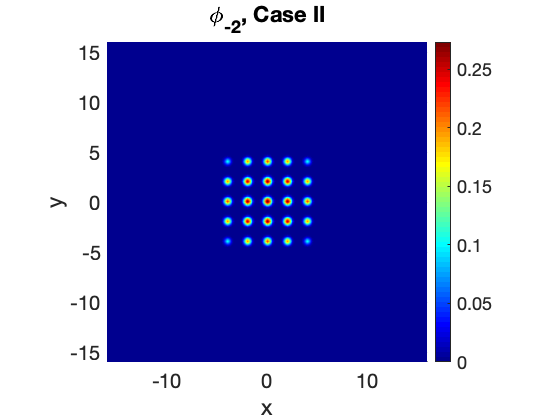}}
\subfigure{\includegraphics[width=0.32\linewidth]{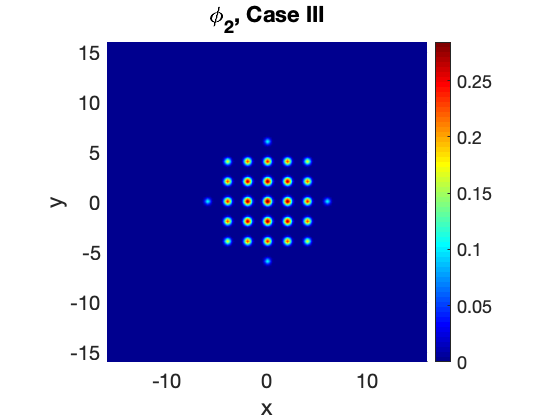}}
\subfigure{\includegraphics[width=0.32\linewidth]{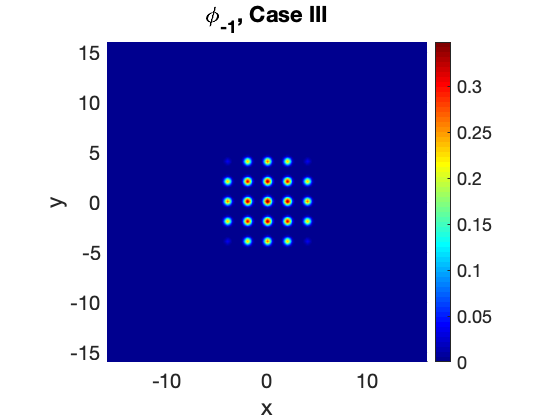}}
\caption{Contour plots for the wave functions of the ground state, i.e. $\phi_2(x,y,0)$, 
$\phi_1(x,y,0)$, $\phi_0(x,y,0)$, $\phi_{-1}(x,y,0)$, $\phi_{-2}(x,y,0)$ of a 
spin-2 BEC in 3D with $M = 0.5$ under different interactions. 
In Case \rn 2, the components $\phi_1(x,y,z)$, $\phi_0(x,y,z)$, $\phi_{-1}(x,y,z)$ 
are zero. In Case \rn 3, the components $\phi_1(x,y,z)$, $\phi_0(x,y,z)$, 
$\phi_{-2}(x,y,z)$ are zero.}
\label{fig:3-3}
\end{figure}

The wave functions of the ground states computed by ARNT are given in Figure \ref{fig:3-1}-
\ref{fig:3-3}. %Similar to spin-1 cases, $\Phi (\mathbf{x})$ rapidly decays to zero when 
%$\| \mathbf{x}\|\rightarrow \infty$. 
The peak function value under ferromagnetic interaction is lower than that under the other two types of interactions. By
comparing Figures \ref{fig:3-1}, \ref{fig:3-2} and \ref{fig:3-3}, we can find a common 
property: when $M>0$, in the ground states under nematic interaction, the components 
$\phi_1, \phi_0, \phi_{-1}$ are always zero-valued functions; and in the ground states 
under cyclic interaction, the components $\phi_1, \phi_0, \phi_{-2}$ are always 
zero-valued functions. 

\subsection{Application in spin-3 BEC}
\label{sect:spin3}

In this section, we apply the ARNT method with the projective retraction to compute the 
ground state of a spin-3 BEC in 1-3 dimensions  under different interactions. In detail, 
the following cases are considered: 
\begin{itemize}
\item 1D, $V(x) = \frac{1}{2} x^2+25\sin^2 (\frac{\pi x}{4})$, $U = [-8, 8], n = 2^8.$
	\begin{itemize}
	\item[~] Case \rn 1. ~~$\beta_0 = 100$, $\beta_1 = 1, \beta_2 = -10, \beta_3 = -1$.
	\item[~] Case \rn 2. ~~$\beta_0 = 100$, $\beta_1 = 1, \beta_2 = 10, \beta_3 = 1$.
	\end{itemize}
\item 2D, $V(x,y)=\frac{1}{2}(x^2+y^2)+10\left[\sin^2 (\frac{\pi x}{2})+
\sin^2 (\frac{\pi y}{2})\right]$, $U = [-8, 8]\times [-8, 8], n = 2^8.$
	\begin{itemize}
	\item[~] Case \rn 1. ~~$\beta_0 = 100$, $\beta_1 = 1, \beta_2 = -10, \beta_3 = -1$.
	\item[~] Case \rn 2. ~~$\beta_0 = 100$, $\beta_1 = 1, \beta_2 = 10, \beta_3 = 1$.
	\end{itemize}
\item 3D, $V(x,y,z)=\frac{1}{2}(x^2+y^2+z^2)+100\left[\sin^2 (\frac{\pi x}{2})+
\sin^2 (\frac{\pi y}{2})+\sin^2(\frac{\pi z}{2})\right]$, $U = [-8, 8]\times [-8, 8]\times 
[-8, 8], n = 2^7.$
	\begin{itemize}
	\item[~] Case \rn 1. ~~$\beta_0 = 100$, $\beta_1 = 1, \beta_2 = -10, \beta_3 = -1$.
	\item[~] Case \rn 2. ~~$\beta_0 = 100$, $\beta_1 = 1, \beta_2 = 10, \beta_3 = 1$.
	\end{itemize}
\end{itemize}

\begin{table}[htb] 
\centering 
\caption{Ground state energies of spin-3 BEC} 
\begin{tabular}{c|cc|cc|cc} 
\hline 
\multirow{2}*{$M$} & \multicolumn{2}{c|}{1D} & \multicolumn{2}{c|}{2D} & \multicolumn{2}{c}{3D} \\ \cline{2-7}~ & Case \rn 1 & Case \rn 2 & Case \rn 1 & Case \rn 2 & Case \rn 1 & Case \rn 2 \\ \hline 
0.0 & 17.1091 & 17.2527 & 11.7811 & 11.8279 & 42.9028 & 42.9752 \\ 
0.5 & 17.1289 & 17.2693 & 11.7877 & 11.8334 & 42.9115 & 42.9822 \\ 
1.5 & 17.2905 & 17.4034 & 11.8415 & 11.8780 & 42.9825 & 43.0399 \\ 
\hline 
\end{tabular} \label{tb:4-1}
\end{table}

\begin{figure} 
\centering 
\subfigure{\includegraphics[width=0.32\linewidth]{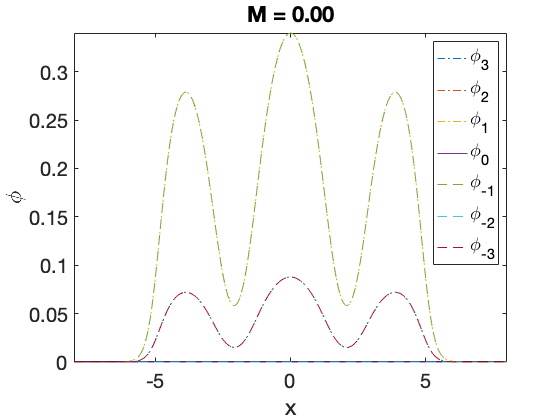}}
\subfigure{\includegraphics[width=0.32\linewidth]{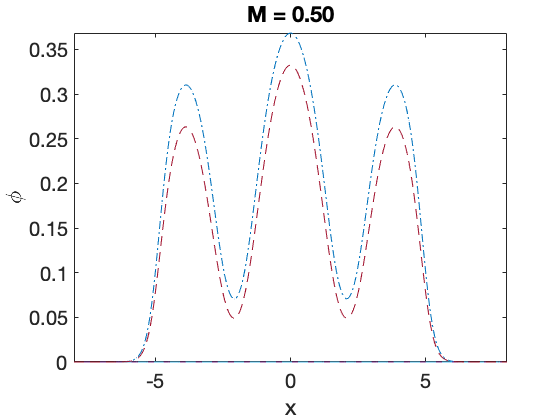}}
\subfigure{\includegraphics[width=0.32\linewidth]{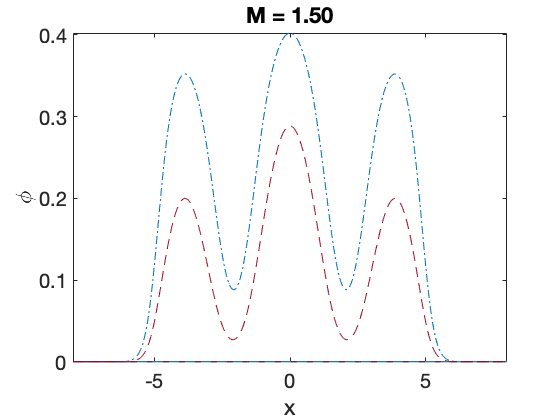}}
\caption{Wave functions of the ground state, i.e. $\phi_3(x)$, $\phi_2(x)$, $\phi_1(x)$,
$\phi_0(x)$, $\phi_{-1}(x)$, $\phi_{-2}(x)$ and $\phi_{-3}(x)$ of a spin-3 BEC for Case \rn 1 
in 1D under different magnetizations $M = 0, 0.5, 1.5$.} 
\label{fig:4-1}
\end{figure}

\begin{figure} 
\centering 
\subfigure{\includegraphics[width=0.32\linewidth]{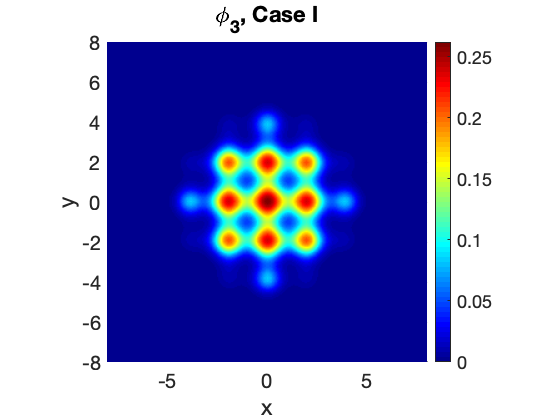}}
\subfigure{\includegraphics[width=0.32\linewidth]{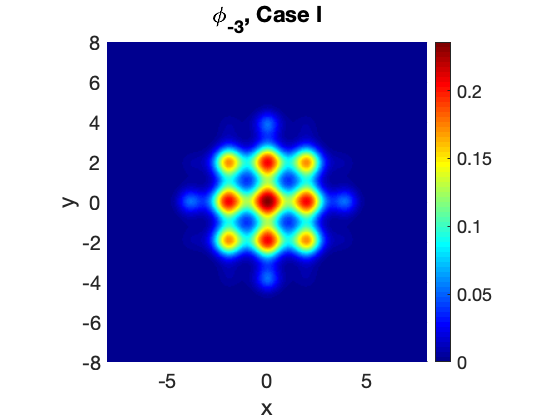}}
\vfill
\subfigure{\includegraphics[width=0.32\linewidth]{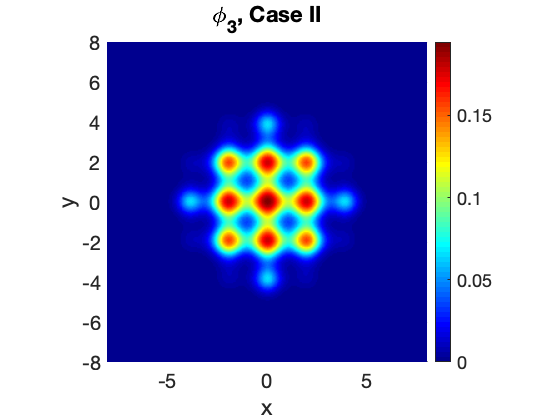}}
\subfigure{\includegraphics[width=0.32\linewidth]{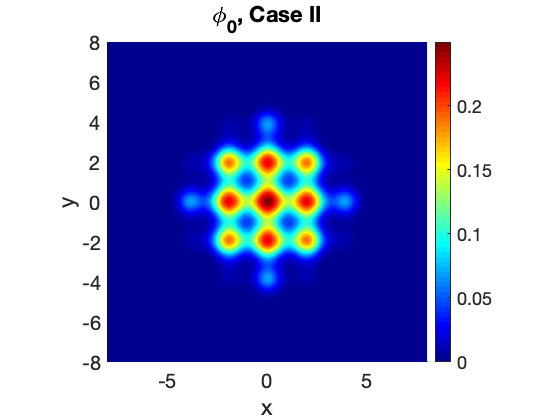}}
\subfigure{\includegraphics[width=0.32\linewidth]{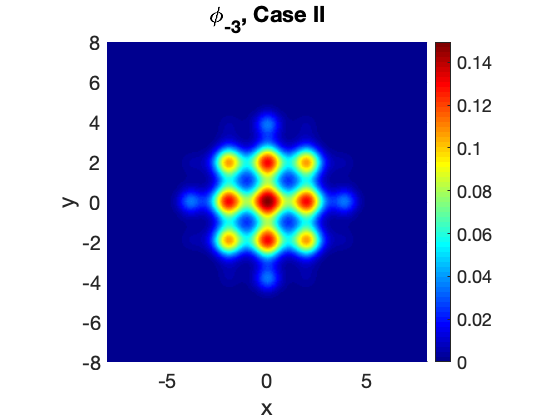}}
\caption{Contour plots for the wave functions of the ground state, i.e. $\phi_3(x,y)$, 
$\phi_2(x,y)$, $\phi_1(x,y)$, $\phi_0(x,y)$, $\phi_{-1}(x,y)$, $\phi_{-2}(x,y)$, 
$\phi_{-3}(x,y)$ of a spin-3 BEC in 2D with $M = 0.5$ under different interactions. 
In Case \rn 1, the components $\phi_2(x,y)$, $\phi_1(x,y)$, $\phi_0(x,y)$, 
$\phi_{-1}(x,y)$, $\phi_{-2}(x,y)$ are close to zero; In Case \rn 2, the components $\phi_2(x,y)$, 
$\phi_1(x,y)$, $\phi_{-1}(x,y)$, $\phi_{-2}(x,y)$ are close to zero.}
\label{fig:4-2}
\end{figure}

\begin{figure} 
\centering 
\subfigure{\includegraphics[width=0.32\linewidth]{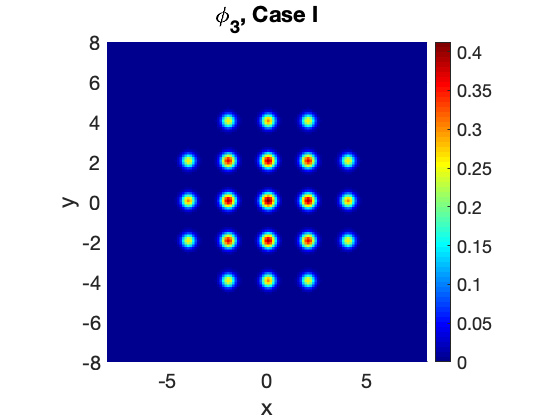}}
\subfigure{\includegraphics[width=0.32\linewidth]{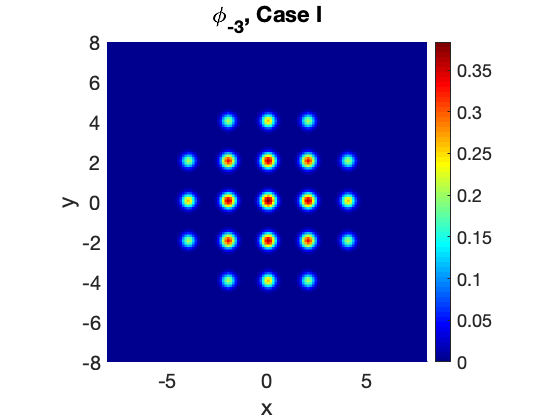}}
\vfill
\subfigure{\includegraphics[width=0.32\linewidth]{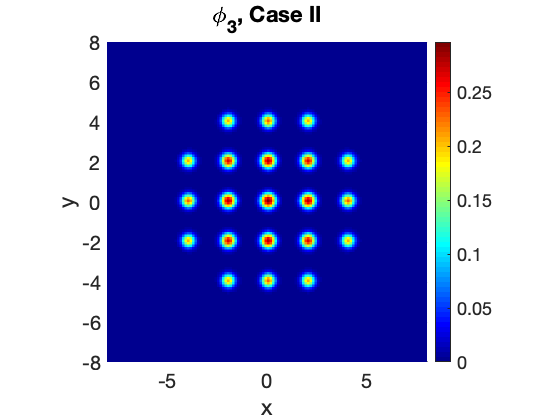}}
\subfigure{\includegraphics[width=0.32\linewidth]{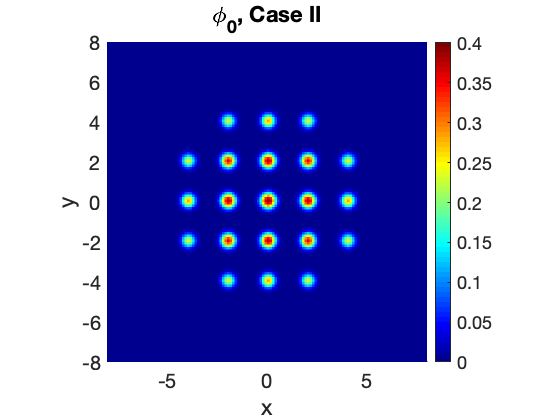}}
\subfigure{\includegraphics[width=0.32\linewidth]{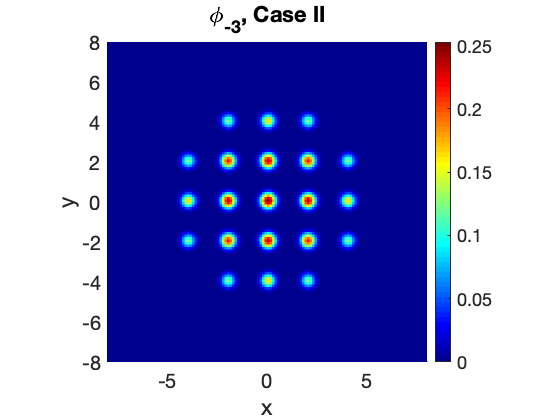}}
\caption{Contour plots for the wave functions of the ground state, i.e. $\phi_3(x,y,0)$, 
$\phi_2(x,y,0)$, $\phi_1(x,y,0)$, $\phi_0(x,y,0)$, $\phi_{-1}(x,y,0)$, $\phi_{-2}(x,y,0)$, 
$\phi_{-3}(x,y,0)$ of a spin-3 BEC in 3D with $M = 0.5$ under different interactions. 
In Case \rn 1, the components $\phi_2(x,y,z)$, $\phi_1(x,y,z)$, $\phi_0(x,y,z)$, 
$\phi_{-1}(x,y,z)$, $\phi_{-2}(x,y,z)$ are close to zero, In Case \rn 2, the components 
$\phi_2(x,y,z)$, $\phi_1(x,y,z)$, $\phi_{-1}(x,y,z)$, $\phi_{-2}(x,y,z)$ 
are close to zero.}
\label{fig:4-3}
\end{figure}

The ground state energies in above cases are listed in Table \ref{tb:4-1}. 
In each case, the energy increases as $M$ increases.
The wave functions of the ground states computed by ARNT are given in 
Figures \ref{fig:4-1}-\ref{fig:4-3}. %Similar to spin-1 and spin-2 cases, 
%$\Phi (\mathbf{x})$ rapidly decays to zero when $\| \mathbf{x}\|\rightarrow \infty$. 
By comparing the figures, we can see that in Case \rn 1, when $M>0$, the components 
$\phi_2, \phi_1, \phi_0, \phi_{-1}, \phi_{-2}$ are always close to zero; in Case \rn 2, 
the components $\phi_2, \phi_1, \phi_{-1}, \phi_{-2}$ are always close to zero
($\infty$-norm less than $10^{-6}$).

\section{Conclusions}
\label{sect:conclusion}

The Fourier pseudospectral method was adopted to discretize the energy functional and constraints for computing the ground states of spin-$F$ Bose-Einstein condensate (BEC). 
The original variational problem was reduced to a finite dimensional Riemannian 
optimization problem. 
An adaptive regularized Newton method, combined with a Riemannian gradient method 
and a cascadic multigrid technique, was designed to solve the discretized problem.
Three different retractions were proposed to implement the optimization algorithms on the manifold. 
Comparison with the Riemannian gradient method and trust-region method for
different retractions and parameters showed that our approach is more efficient and stable. 
Extensive numerical examples of spin-2 and spin-3 BEC in 1D, 2D and 3D with 
optical lattice potential and various interaction demonstrated the robustness 
of our approach. The energy and wave functions of ground states are reported to reveal some interesting physical phenomena. 
Our method is the first one to explore spin-3 BEC computationally. Although 
the spin-3 cases discussed in this paper are relatively simple, our algorithm 
is also applicable for cases with more diverse parameters.

\bibliographystyle{siam}
\bibliography{BEC}

\end{document}